\documentclass[a4paper]{amsart}
\usepackage{amsmath,amsthm,amsfonts,amssymb, color}
\usepackage{amsmath}%
\usepackage{amsfonts}%
\usepackage{amssymb}%
\usepackage{graphicx}
\usepackage{srcltx}
\usepackage{dsfont}
\usepackage[latin1]{inputenc}
\usepackage{graphicx}
\usepackage{color}
\usepackage{caption}
\numberwithin{equation}{section}
\theoremstyle{plain}

\newtheorem{theorem}{\sc Theorem}[section]

\newtheorem{lemma}[theorem]{\sc Lemma}
\newtheorem{proposition}[theorem]{\sc Proposition}

\usepackage{accents}

\newcommand{\be}{\begin{equation}}
\newcommand{\ee}{\end{equation}}

\providecommand{\norm}[1]{\Vert#1\Vert}

\newcommand{\N}{\mathbb{N}}
\newcommand{\xk}[1]{x_{#1}}
\newcommand{\kix}[1]{\xi_{#1}}

\newcommand{\lam}[1]{\lambda_{#1}}

\newcommand{\incl}[1]{\mathds{1}\{{#1}\}}
\newcommand{\lamm}[1]{\lambda_{{#1}}}
\newcommand{\E}{\mathbb{E}}
\newcommand{\nul}[1]{\nu_{{#1}}}
\newcommand{\tn}{2^n}

\def\e{\varepsilon}

\begin{document}

\title[The Parabolic Anderson model on the hypercube]{The Parabolic Anderson model on the hypercube}

\author{Luca Avena}
\address{Luca Avena\\Universiteit Leiden\\ Niels Bohrweg 1
	\\ 2333 CA Leiden\\Netherlands}
\email{l.avena@math.leidenuniv.nl}

\author{Onur G\"un}
\address{Onur G\"un\\ Weierstrass Institute\\ Mohrenstrasse 39\\ 10117 Berlin\\ Germany}
\email{Onur.Guen@wias-berlin.de}

\author{Marion Hesse}
\address{Marion Hesse\\ Weierstrass Institute\\ Mohrenstrasse 39\\ 10117 Berlin\\ Germany}
\email{Hesse@wias-berlin.de}

\date{\today}


\maketitle

\begin{abstract}
We consider the parabolic Anderson model $\frac{\partial}{\partial t} v_n=\kappa\Delta_n v_n + \xi_n v_n$ on the $n$-dimensional hypercube $\{-1,+1\}^n$ with random i.i.d. potential $\xi_n$. We parametrize time by volume and study $v_n$ at the location of the $k$-th largest potential, $x_{k,\tn}$. Our main result is that, for a certain class of potential distributions, the solution exhibits a phase transition: for short time scales $v_n(t_n,x_{k,\tn})$ behaves like a system without diffusion and grows as $\exp\big\{(\xi_n(x_{k,\tn}) - \kappa)t_n\big\}$, whereas, for long time scales the growth is dictated by the principle eigenvalue and the corresponding eigenfunction of the operator $\kappa \Delta_n+\xi_n$, for which we give precise asymptotics. Moreover, the transition time depends only on the difference $\xi_n(x_{1,\tn})-\xi_n(x_{k,\tn})$. 

One of our main motivations in this article is to investigate the mutation-selection model of population genetics on a random fitness landscape, which is given by the ratio of $v_n$ to its total mass, with $\xi_n$ corresponding to the fitness landscape. We show that the phase transition of the solution translates to the mutation-selection model as follows: a population initially concentrated at $x_{k,\tn}$ moves completely to $x_{1,\tn}$ on time scales where the transition of growth rates happens. The class of potentials we consider involve the Random Energy Model (REM) of statistical physics which is studied as one of the main examples of a random fitness landscape.

\end{abstract}

\section{Introduction and main results}

\subsection{The Model.}
Consider the $n$-dimensional hypercube $\Sigma_n=\{-1,1\}^n$, $n\in\N$. For $x\in\Sigma_n$, we use the notation $x=(x^{(1)},\dots,x^{(n)})$, where $x^{(i)}$ denotes the spin of $x$ at spin site $i$. The Hamming distance on $\Sigma_n$ is defined by
\begin{equation}
d(x,y)=\#\{i:\;x^{(i)}\not= y^{(i)}\}.
\end{equation}
We declare that $x$ and $y$ are neighbours, denoted by $x\sim y$, if $d(x,y)=1$.

Our model is described through a system of differential equations with random potential,
\begin{equation}\label{PAM}
\frac{\partial}{\partial t} v_n(t,x,y)=\kappa\Delta_n v_n(t,x,y) + \xi_n(x) v_n(t,x,y) ,\quad t\geq 0,\; x\in\Sigma_n
\end{equation}
with the localized initial condition $v_n(0,\cdot,y)= \delta_y(\cdot)$. Here, $\kappa>0$ is the diffusion constant and $\Delta_n$, acting on the second coordinate, denotes the Laplace operator on $\Sigma_n$
\begin{equation}
\Delta_n f(x):=\frac{1}{n}\sum_{z\sim x}\big(f(z)-f(x)\big),\quad x\in\Sigma_n,
\end{equation} 
where $f$ is a function on $\Sigma_n$ and ${\xi}_n:=\{\xi_n(x):\;x\in\Sigma_n\}$ is the random potential. 

The solution of (\ref{PAM}) admits a Feynman-Kac representation
\begin{equation}\label{feynmankac}
v_n(t,x,y)=\mathbb{E}_{x}[\exp(\int_0^t \xi_n(X_s)ds)\incl{X_t=y}]
\end{equation}
where $(X_s:s\geq 0)$ is distributed as a simple random walk on $\Sigma_n$ with the generator $\kappa \Delta_n$ and $\mathbb{E}_x$ stands for its expectation when the walk starts at $x$, i.e., $X_0=x$. Since the simple random walk on $\Sigma_n$ is time reversible, we can conclude from (\ref{feynmankac}) that
\begin{equation}
v_n(t,x,y)=v_n(t,y,x).
\end{equation}
We also deal with the de-localized model. Let $v_n(t,\cdot)$ be the solution of (\ref{PAM}) with the initial condition $v_n(0,\cdot)\equiv 1$. It is trivial that $v_n(t,y)=\sum_{x\in\Sigma_n}v_n(t,x,y)$ and $v_n(t,y)$ admits the Feynman-Kac representation
\begin{equation}
v_n(t,y)=\mathbb{E}_{y}[\exp(\int_0^t \xi_n(X_s)ds)].
\end{equation}
Equation (\ref{PAM}) and its variants are often called the parabolic Anderson model. PAM originates as the parabolic version of the Anderson localization problem and has found a wide range of applications such as chemical kinetics, magnetism, turbulence and population dynamics, the last being one of the motivations of this article. PAM is also attractive for mathematicians since it yields precise solutions based on the Feynman-Kac representation and the spectral analysis of the Hamiltonian operator $\kappa\Delta+\xi$. We refer the readers to the recent book \cite{W16} and the references therein for the applications of PAM and a survey of mathematical results. The main feature of (\ref{PAM}) is the competition between the diffusion term that flattens the solution and the potential part that creates peaks. A feature of this competition is the \textit{intermittency effect}, namely, the total mass of the solution $v_n$ is carried by a few separated regions with small diameters. Indeed, the rigorous mathematical research on PAM started with the seminal paper \cite{GM90} in which intermittency was proved under minimal conditions on the potential. In a follow-up paper \cite{GM98} the same authors gave a description of the shape of relevant islands in terms of a variational problem in the growth rate. The potentials considered in \cite{GM98} consisted of distributions with upper tails that are double exponential or slightly heavier/lighter. The asymptotic size of the islands are finite for double exponential tails that and shrinks to a single site for heavier tails. This geometric picture was made precise in \cite{GKM07}. The growth of the solution for much heavier tails has random first order terms, and results in this direction was achieved in \cite{HMS08} for potentials with Pareto and exponential distributions. Later, \cite{KLMS09} proved single site localization for the same kind of potentials and The evolution of the localization point was investigated in the context of aging in \cite{MOS11}.

In this work we consider PAM on the $n$-dimensional hypercube for class of potentials that lead to single site localization. We will describe the growth of the solution and provide localisation results. Our point of view focuses on solutions starting from the site of an extremal potential and how the growth and localisation change with time, observing a phase transition in time. Moreover, we will explore the fact the normalized solution of PAM corresponds to a mutation-selection model and explain our localisation results in term of the latter.

We want to mention that the state space for all the results we have mentioned from the literature is the $d$-dimensional lattice. There are only a few work about PAM on different graphs, one being \cite{FM90} where the authors study PAM on complete graph with exponentially distributed random potential. This work has been an inspiration for us as it also proves a phase transition on the growth depending on the time scales. One big simplicity of working on the complete graph is that the exact asymtotics of the whole spectrum of eigenvalues and eigenvectors is readily available.

Let us now briefly explain the results of this article. Our first main result is an exact description of the behaviour of the solution at the location of the $k$-th largest potential. Let $x_{1,\tn},x_{2,\tn},\dots,x_{\tn,2^n}$ denote the locations of the largest potential, second largest potential and so on. We denote by $\lambda_{1}$ the principle eigenvalue of the operator $\kappa \Delta_n+\xi_n$. For the potential we essentially assume that, for any $k\in$ fixed, almost surely $\xi_n(x_{k,\tn})\sim \theta n$ for some $\theta>0$ and the gap between the extremal points $\xi_n(x_{1,\tn})-\xi_n(x_{k,\tn})$ stays order of (random) constant (see Section 1.2). The behaviour of $v_n(t_n,x_{k,\tn})$ goes through a transition on time scales of order $n\log n$. To this end let $c_n=\frac{1}{2} n\log n$ and let $t_n/c_n\to\alpha \in [0,\infty]$. We prove that (see Section \ref{sec1.4}) for $\alpha<(\xi(x_{1,\tn})-\xi(x_{k,\tn}))^{-1}$
\begin{equation}
v_n(t_n,x_{k,\tn})\sim \exp\big\{(\xi(x_{k,\tn})-\kappa)t_n\big\},
\end{equation}
and for $\alpha > (\xi(x_{1,\tn})-\xi(x_{k,\tn}))^{-1}$ 
\begin{equation}\label{second}
v_n(t_n,x_{k,\tn})\sim \exp\big\{\lambda_1 t_n\big\}\exp\big\{-c_n(1+o(1))\big\}.
\end{equation}
Hence, in the short time regime the solution at all the high peaks grows like a system without diffusion, more precisely, when the potential is shifted down by $\kappa$ and the diffusion is removed. However, on the long time regimes, with the observation that (see Lemma \ref{lemmaAna2})
\begin{equation}
\lambda_1=\xi_{1,\tn}-\kappa+O(1/n^2),
\end{equation}
we see that the growth is much higher. We also mention that the second term in (\ref{second}) is the decay of the principle eigenfunction at $x_{k,\tn}$. The class of potentials we consider involves the Random Energy Model (REM) of spin glasses introduced in \cite{D81}, where the potential field is formed by i.i.d. Gaussian random variables with mean 0 and variance $n$.

Now we describe the mutation-selection model on the hypercube with random fitness landscape and explain how it is connected to PAM.  The mutation-selection model is given by the solution $u_n(t,\cdot,y)$ of the following PDE
\begin{equation}\label{mut-sel}
u_n(t,x,y)=\kappa\Delta_n u_n(t,x,y) + \left[\xi_n(x) -\bar{\xi}(t)\right]u_n(t,x,y) ,\quad t\geq 0,\; x\in\Sigma_n
\end{equation}
with the localized initial condition $u_n(t,\cdot,y)=\delta_y(\cdot)$,
where $\bar{\xi}(t)$ is the mean fitness
\begin{equation}\label{meanfitness}
\bar{\xi}(t):=\sum_{x\in\Sigma_n}u_n(t,x,y)\xi_n(x).
\end{equation}
Let us briefly explain the biological meaning of the mathematical objects appearing in (\ref{mut-sel}). Haploid genotypes are identified with linear arrangement of $n$ sites $x=(x^{(1)},\dots,x^{(n)})$ with each site taking values $-1$ or $+1$. In the multilocus context sites correspond to \textit{loci} and the variables $x^{(i)}$ to \textit{alleles}. In the context of molecular evolution, $x$ corresponds to a DNA (or RNA) sequence where the nucleotides are lumped into purines (say, $+1$) and pyrimidines (say, $-1$). In biology literature the hypercube $\Sigma_n$ is usually called the sequence space. Then the mutation-selection model given in (\ref{mut-sel}) describes the evolution of an infinite population of haploids that experience only mutation and selection. The population evolves in continuous time (non-overlapping generations) with mutation and selection occurring independently (parallel). $\xi_n(x)$ is the Malthusian fitness of \textit{type} $x$ and form a fitness landscape, which in our case is random. Site mutations happen with rate $\kappa/n$ (hence, a total rate of $\kappa$). From (\ref{meanfitness}) it follows that $\sum_x u_n(t,x,y) =1$, and $u_n(t,x,y)$ corresponds to the frequency of type $x$ under this evolution. Finally, note that the localized initial condition means that initially the population consists of only type $y$. The competition between diffusion and potential discussed in PAM translates as competition between mutation and selection, two driving forces of Darwinian evolution. The mutation-selection model dates back to Wright \cite{W32}. We refer readers to the classical book \cite{CK70} for an introduction to population genetics and to \cite{BG99} for an excellent survey that involves the statistical physics methods used to solve mutation-selection models for a wide range of landscapes. 

The motivation to consider a random fitness landscape is the following. Realistic landscapes are expected to be complex with structures such as valleys and hills \cite{VK14}.  Random fitness landscapes naturally form a class of complex landscapes. The first obvious choice, that is an i.i.d.~landscape, is also known as the {\em House of Cards model} and was introduced by Kingman \cite{K78}.

It is well-known that (see e.g. \cite{M76}) the linear system $v_n$ can be transformed to $u_n$ via normalization by its total mass, that is,
\begin{equation}
u_n(t,x,y)=\frac{v_n(t,x,y)}{v_n(t,y)}.
\end{equation}
In a way $v_n(t,x,y)$ can be thought as \textit{absolute} frequencies.
Hence, behaviour of the mutation-selection model is related to the localization properties of the PAM model. Indeed, we will prove that (see Section \ref{sec1.3}) the phase transition occurring exhibited in growth rates of $v_n$ translates to the behaviour of $u_n$. Namely, on short time scales $u_n(t_n,x_{k,\tn},x_{k,\tn})\to 1$, whereas, on long time scales $u_n(t_n,x_{1,\tn},x_{k,\tn})\to 1$. In other words, a population initially consisted of $x_{k,\tn}$ type individuals stays that way for a certain threshold in time, after which it is invaded by the best fit type $x_{k,\tn}$. 

The coupled model where the reproduction events are followed by mutation is known as quasispecies model, introduced by Eigen in \cite{Eig71}. Main feature of this model is the existence of a \textit{error threshold}, that is, for a {\it single peak} landscape (a \textit{master} sequence has a fitness $\sigma>1$ and the rest has the same fitness of 1) in the limit as the genome length $n\to\infty$ and time $t\to\infty$ the population is essentially randomly distributed over the space if the mutation rate is above a certain value, whereas for the mutation rates below this critical value the population consists of individuals close to the master sequence, what Eigen calls a \textit{quasispecies}. Similar results were proven in \cite{FPS93} and \cite{FP97} for the REM landscape. We have to emphasize that our model is actually not in the direction of these results. In the quasispecies models we have mentioned the mutation rate and the fitness at highest peak is on the same scale. In our case the fitness of the highest peak is on the scale of $n$ while the mutation rate is kept at constant. Hence, we do not have the quasispecies picture. Instead, what we focus on is studying the evolution in intermediate time scales, that is, before the equilibrium. The phase transition we observe is on the time scale of observation rather than on the mutation rate.   

In the rest of this section we describe precisely the potentials we use, then we state our main results on the growth rates and localisation, and finally we quickly show that REM landscape satisfies our assumptions on the potential field.

{\it Notation.} Throughout the paper we use the notations $o,O,\Theta,\ll,\gg,\sim$ for any two sequences $f_n,g_n$ as follows. We write $f_n=o(g_n)$, $f_n\ll g_n$ or $g_n\gg f_n$ if $f_n/g_n\rightarrow 0$ as $n\to\infty$; $f_n=O(g_n)$ if $\limsup_{n\to\infty}f_n/g_n<\infty$; $f_n=\Theta(g_n)$ if there are positive constants $C_1,C_2$ such that $C_1 g_n\leq f_n\leq C_2 g_n$ for all $n$ large enough; and $f_n\sim g_n$ if $f_n/g_n\to 1$ as $n\to \infty$. Moreover, For constants in our estimates we use the letter $C$ freely as long as it does not appear at the end result.
\subsection{The potential}\label{sec1.2}

\newcommand{\bxi}{{\xi}_n}
\newcommand{\bota}{{\eta}_n}

For each $n$, ${\xi}_n=\{\xi_n(x):x\in\Sigma_n\}$ is a collection of i.i.d. random variables whose common cumulative distribution function is denoted by $G_n$. We assume that $G_n$ is continuous, i.e. $\xi_n(x)$ has no atoms. We define
\begin{equation}
\varphi_n(r):=\log \frac{1}{1-G_n(r)},\quad  r\in\mathbb{R},
\end{equation}
and its left-continuous inverse
\begin{equation}
\psi_n(s):=\min\{r:\varphi_n(r)\geq s\},\quad s>0.
\end{equation}
Let ${\eta}_n=\{\eta_n(x):x\in\Sigma_n\}$ be an i.i.d. field of mean 1 exponential random variable. Then $\psi_n(\eta_n)\overset{d}=\bxi$, and from now on we assume without loss of generality that $\bxi=\psi_n(\bota)$. Note that since $\psi_n$ is strictly increasing the sites ordered according to their potentials coincide for the two fields. More precisely, we can label the vertices of $\Sigma_n$ by $x_{1,2^n},\dots,x_{2^n,2^n}$ so that
\begin{equation}
\xi_n(x_{1,2^n}):=\xi_{1,2^n}>\xi(x_{2,2^n}):=\xi_{2,2^n}>\cdots >\xi_n(x_{2^n,2^n}):=\xi_{2^n,2^n}
\end{equation}
and
\begin{equation}
\eta_n(x_{1,2^n}):=\eta_{1,2^n}>\eta_n(x_{2,2^n}):=\eta_{2,2^n}>\cdots >\eta_n(x_{2^n,2^n}):=\eta_{2^n,2^n}.
\end{equation}
Note that the above inequalities are strict because $G_n$ is continuous. Let $\sigma_i$, $i\in\N$, be an independent sequence of random variables where $\sigma_i$ is exponentially distributed with intensity $i$. It is well-known that (see e.g. Section I.6 of \cite{F66}) 
\begin{equation}
(\eta_1,\eta_2,\dots,\eta_{2^n})\overset{d}=(\sigma_1+\cdots+\sigma_{2^n},\sigma_2+\cdots+\sigma_{2^n},\dots, \sigma_{2^n}).
\end{equation}
From now on we describe the field $\bota$ (and in turn the field $\bxi$) through the sequence $(\sigma_i,\; i\in\N)$. Namely, $\bota$ is given by its order statistics $(\eta_{1,2^n},\eta_{2,2^n},\dots,\eta_{2^n,2^n})$ coupled to $(\sigma_i,\; i\in\N)$ via $\eta_{i,2^n}=\sigma_i+\cdots+\sigma_{2^n}$. We denote by $P$ and $E$ the distribution and expectation in this common probability space, respectively. 

Since
\begin{equation}
P(\eta_{1,2^n}\geq nc\log 2 )\leq 2^{-n(c-1)},\quad \forall c>1,
\end{equation}
and
\begin{equation}
P(\eta_{1,2^n}\leq n c\log 2 )\leq \exp(-2^{n(1-c)}),\quad \forall c<1,
\end{equation}
by an application of Borel-Cantelli lemma $P$-a.s.
\begin{equation}\label{expfieldgrow}
\lim_{n\to\infty} \frac{\eta_{1,2^n}}{n}=\log 2.
\end{equation}
We have $\eta_{1,2^n}-\eta_{k,2^n}=\sigma_1+\cdots+\sigma_{k-1}$. Hence, $P$-a.s. for any $k\in\N$
\begin{equation}
\lim_{n\to\infty} \frac{\eta_{k,2^n}}{n}=\log 2.
\end{equation}
Therefore, the extremes of the field $\bota$ all grow like $n\log 2 $ and the gap between extremal points are (random) constants, i.e., 
\begin{equation}\label{expgap}
\eta_{k,2^n}-\eta_{l,2^n}=\sigma_k+\cdots\sigma_{l-1}, \quad \mbox{ for any }  k<l.
\end{equation} 

We now list our assumptions for the field $\bxi$. The first set of assumptions is about the extremes of the field and concerns only the right-tail of the distribution of $\xi_n$ in terms of $\psi_n$. The following assumption identifies the growth rate of the extremes.

{\bf Assumption $(R1)$}
For any $a>0$
\begin{equation}
\psi_n(an)\sim f(a)n
\end{equation}
where $f:\mathbb{R}^+\to\mathbb{R}^+$ is a strictly increasing function. We define $\theta:=f(\log 2)$.

Hence, by (\ref{expfieldgrow}), $\xi_{1,2^n}$ grows like $\theta n$. The choice of this growth rate is arbitrary but it makes the representation cleaner and this is the actual case for REM. 

Our second assumption on the right-tail is more crucial, it guarantees that, like in the exponential field, the gaps between extremes are order of (random) constants.

{\bf Assumption $(R2)$}
For any sequence $s_n\sim \theta n$, for any $c\in\mathbb{R}$,
\begin{equation}
\psi(s_n+c)-\psi(s_n)\to g(c),
\end{equation}
where $g:\mathbb{R}\to\mathbb{R}$ is such that $g(c)\not= 0$ for any $c\not= 0$.

Therefore, by (\ref{expfieldgrow}) and (\ref{expgap}), $P$-a.s.
\begin{equation}
\xi_{k,\tn}-\xi_{l,\tn}=g(\sigma_{k}+\cdots+\sigma_{l-1})+o(1), \quad \mbox{ for any }  k<l.
\end{equation}
Recall that $g(c)>0$ for $c>0$, that is, the gap above does not vanish. For convenience we define
\begin{equation}
\xi_{k,l}:=g(\sigma_k+\cdots+\sigma_{l-1}), \mbox{ for } k<l,
\end{equation}
and set $\xi_{k,l}=0$ for $k\geq l$. Note that $\xi_{k,l}$ does not depend on $n$.

For further reference, we sum up the implications of Assumptions $(R1)$ and $(R2)$ in a lemma
\begin{lemma}\label{lemma1}
	Let assumptions $(R1)$ and $(R2)$ be satisfied. Then $P$-a.s. for any $k,l\in\N$
	\begin{itemize}
		\item[(i)]
		\begin{equation}
		\lim_{n\to\infty}\frac{\xi_{k,\tn}}{n}=\theta;
		\end{equation}
		\item[(ii)]
		\begin{equation}
		\xi_{k,\tn}-\xi_{l,\tn}=\xi_{k,l}+o(1), \quad \mbox{ for any }  k<l.
		\end{equation}
	\end{itemize}
\end{lemma}

Our last assumption concerns the left tail of the distribution of $\bxi$.

{\bf Assumption $(L)$}
There exists a sequence $l_n\ll n$ for which
\begin{equation}
\sum_{n\in \N}nG_n(-l_n)<\infty.
\end{equation}

Essentially, above assumption yields that there are enough path between extremal points that avoid sites with potential smaller than $-l_n$. Moreover, it guarantees that the neighbours of extremal points also have potential not smaller than $-l_n$. 

Now we are ready to formulate our results rigorously.

\subsection{Growth Rates}\label{sec1.4}

Let 
\begin{equation}
c_n:=\frac{1}{2}n\log n,
\end{equation}
and consider time scales $t_n$ such that
\begin{equation}
\frac{t_n}{c_n}\underset{n\to\infty}\longrightarrow \alpha\in[0,\infty].
\end{equation}
We denote by $\lambda_{1}$ the principle eigenvalue of the operator $\kappa \Delta_n+\xi_n$. Note that with a slight abuse of notation we do not use $n$ in $\lambda_1$.
\begin{theorem}\label{thrmGrow}
	Let Assumptions $(R1), (R2)$ and $(L)$ be satisfied. Then $P$-a.s. for any $k\in\N\setminus\{1\}$ as $n\to\infty$
	\begin{equation}
	v_n(t_n,\xk{k,2^n})\sim \left\{
	\begin{array}{ll}
	\exp\Big\{\big(\xi_{k,\tn}-\kappa\big) t_n\Big\} & \mbox{ if }\alpha<1/\xi_{1,k},\\
	&\\
	\exp\Big\{\lam{1}t_n-c_n(1+o(1))\Big\}& \mbox{ if }\alpha>1/\xi_{1,k}.
	\end{array}\right.
	\end{equation}
	Moreover, for any $\alpha\in[0,\infty]$
	\begin{equation}\label{thrmGrowx1}
	v_n(t_n,\xk{1})\sim \exp\big\{{\lam{1}t_n}\big\}.
	\end{equation}
\end{theorem}
So on the short time scales the solution grows by the single peak, which can be seen as the model with no diffusion and potential $\xi_{k,2^n}-\kappa$, and on the other hand, for longer time scales the growth is larger which is determined by the principle eigenvalue and a correction term given by the decay of the principle eigenfunction at $x_{k,2^n}$ (see Lemma \ref{lemmaAna} below). We also mention that (see Lemma \ref{lemmaAna2} below)
	\begin{equation}
	\lam{1}=\xi_{1,2^n}-\kappa+\Theta(1/n^2).
	\end{equation}
Let us consider the time scale of phase transition and for simplicity take $t_n=\alpha c_n$. Then the growth rate, to be precise the ratio of the term in the exponentials to the time scale $t_n$, is $\xi_{k,2^n}-\kappa$ for $\alpha<1/\xi_{1,k}$ and $\xi_{1,2^n}-\kappa-1/\alpha+o(1)$ for $\alpha>1/\xi_{1,k}$. Since $\xi_{1,2^n}-\xi_{k,2^n}=\xi_{1,k}+o(1)$ the phase transition is second order, see Figure \ref{plotgrowth}.

\begin{figure} [ht]
	\begin{center}
		\centering
	\includegraphics[scale=0.8, width=8cm]{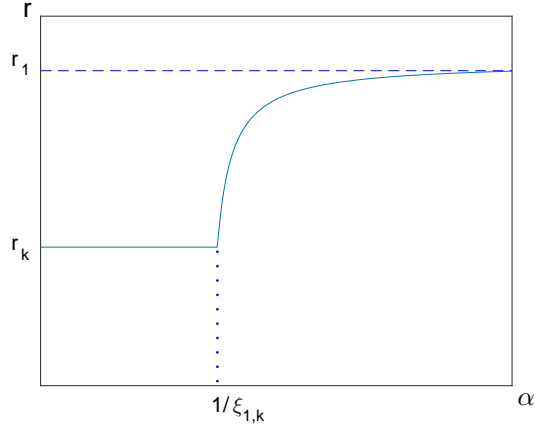}
	\end{center}
	\caption{Plot of growth rate with respect to time. $\alpha$-axis is time normalized by $c_n$, that is, $t_n=\alpha c_n$. $r$-axis is the growth rate shifted by $\theta n$. Here $r_1=\xi_{1,2^n}-\kappa$ and $r_k=\xi_{k,2^n}-\kappa$.}
	\label{plotgrowth}
\end{figure}

\subsection{Localization }\label{sec1.3}

\begin{theorem}\label{thrmLoc}
	Let Assumptions $(R1), (R2)$ and $(L)$ be satisfied. Then $P$-a.s. for any $k\in\N\setminus\{1\}$ as $n\to\infty$
	\begin{itemize}
		\item[(i)] if $\alpha<1/\xi_{1,k}$
		\begin{equation}
		u_n(t_n,\xk{k,\tn},\xk{k,\tn})\longrightarrow 1;
		\end{equation}
		\item[(ii)] if $\alpha> 1/\xi_{1,k}$
		\begin{equation}
		u_n(t_n,\xk{1,\tn},\xk{k,\tn})\longrightarrow 1.
		\end{equation}
	\end{itemize}
	Moreover, for any $\alpha\in[0,\infty]$ $P$-a.s. as $n\to\infty$
	\begin{equation}\label{thrmLocx1}
	u_n(t_n,\xk{1,\tn},\xk{1,\tn})\longrightarrow 1.
	\end{equation}
\end{theorem}

\subsection{REM landscape} Our main application is the REM landscape, that is, ${\xi}_n$ is a collection of i.i.d. Gaussian random variables with variance $n$. 
\begin{proposition}
	The REM landscape satisfies Assumptions $(R1), (R2)$ and $(L)$.
\end{proposition}
\begin{proof}
Let $Z$ denote a standard normal random variable. Then
\begin{equation}
1-G_n(r)=P(Z\geq r/\sqrt{n})=\frac{1}{2\sqrt{\pi}}\int_{r/\sqrt{n}}^{\infty}e^{-{x^2/2}}dx.
\end{equation}	
We use the following trivial bounds
\begin{equation}
e^{-r^2/2}\Big(\frac{1}{r}-\frac{1}{r^3}\Big)\leq \int_{r}^{\infty}e^{-{x^2/2}}dx \leq e^{-r^2/2}\frac{1}{r},\quad \forall r>0. 
\end{equation}
Using the above and the definitions of $\varphi_n$ and $\psi_n$, we get that for any sequence $s_n\sim an$, $a>0$ and $c\in\mathbb{R}$
\begin{equation}\label{inqgauss}
\psi_n(s_n+c)=\sqrt{2ns_n}+\frac{c}{\sqrt{2s_n/n}}-\frac{\log \sqrt{2\pi}}{\sqrt{2s_n/n}}-\frac{\log \sqrt{2s_n}}{\sqrt{2s_n/n}}+o(1).
\end{equation}
Setting $c=0$, Assumption $(R1)$ follows with $f(a)=\sqrt{2a}$. Then by definition $\theta=\sqrt{2\log 2}$. Moreover, 
\begin{equation}
\psi_n(s_n+c)-\psi(s_n)=c/\sqrt{2a}+o(1),
\end{equation}
which yields, for $a=\theta$, Assumption $(R2)$ with $g(c)=c/\sqrt{2\theta}$. Using (\ref{inqgauss}) we have
\begin{equation}
G_n(-l_n)\leq e^{-l_n^2/(2n)}.
\end{equation}
Setting $l_n=n^c$, for some $c\in (1/2,1)$, Assumption $(L)$ is satisfied.
\end{proof}

In the rest of this paper we prove the main results in Section (\ref{sec2}) using two lemmas on spectral properties of the operator $\kappa \Delta_n+\xi_n$, which are proved in Section (\ref{sec3}).

\section{Proof of Theorem \ref{thrmLoc} and Theorem \ref{thrmGrow}}\label{sec2}

We describe the growth of $v_n(t_n,\cdot,x_{k,\tn})$ by using spectral properties of the operator $\kappa \Delta_n +\xi_n$ with zero boundary conditions on certain vertices of extremal potential. To this end we have two main ingredients: firstly, precise descriptions of principle eigenvalue, spectral gap and localization of the principle eigenvector for the aforementioned operators; secondly, a general mechanism allowing us to turn these spectral properties to estimates for $v_n$. For the latter, we follow the general framework established in \cite{GKM07}. 

We introduce the spectral objects we use for our estimates. For $l\in\N$ set $\Gamma_l=\{\xk{1,\tn},\dots,\xk{l,\tn}\}$. For $\xk{i,\tn}\in\Gamma_l$, $i\in\{1,\dots,l\}$, consider the principle eigenvalue and (positive) eigenfunction of the operator $\kappa \Delta_n+\xi_n$ with zero boundary conditions on $\Gamma_l\setminus\{\xk{i,\tn}\}$, denoted by $\lamm{i,l}$ and $\nul{i,l}$, respectively, where $\nul{i,l}$ is normalized so that $\nul{i,l}(\xk{i,\tn})=1$. Let $g_{i,l}$ denote the corresponding spectral gap, that is, the difference between the principle eigenvalue and the second largest eigenvalue. We write $\lambda_i$ and $\nu_{i}$ for $\lambda_{i,i}$ and $\nu_{i,i}$, respectively. Note that as before we do not use $n$ in the notation for eigenvalues and eigenvectors. Finally, for $A\subseteq \Sigma_n$ we define the hitting time
\begin{equation}
\tau_A:=\inf\{t\geq 0:\;X_t\in A\},
\end{equation}
and write simply $\tau_x$ for $\tau_{\{x\}}$.

We have a probabilistic representation for $\nul{i,l}$ given by
\begin{equation}\label{repeigvec}
\nul{i,l}(x)=\E_x\left[\exp\Big(\int_0^{\tau_{\xk{i,\tn}} } \big[\xi_n(X_s)-\lamm{i,l}\big]ds  \Big)\incl{\tau_{\xk{i,\tn}} \leq \tau_{\Gamma_{l}}  }\right].
\end{equation}

The following two lemmas contain the main spectral results. We postpone their proof until the next section.
\begin{lemma}\label{lemmaAna2}
	Let Assumptions $(R1), (R2)$ and $(L)$ be satisfied. Then $P$-a.s. for any $k\in\N$ for all $1\leq i \leq l\leq k$
	\begin{equation}
	\lamm{i,l}=\kix{i,\tn}-\kappa+\Theta(1/n^2)
	\end{equation}
	and 
	\begin{equation}
	\quad  g_{i,l}\geq \kix{i,\tn}-\kix{l+1,\tn}.
	\end{equation}
\end{lemma}
\begin{lemma}\label{lemmaAna}
	Let Assumptions $(R1), (R2)$ and $(L)$ be satisfied. Then $P$-a.s. for any $k\in\N$ the followings are true:
	\begin{itemize}
		\item[(i)]
$
    \sum_{\;x\not=\xk{i}} \nul{i,l}(x)\longrightarrow 0
$ and $\norm{\nul{i,l}}_2^2\longrightarrow 1$,  for all $1\leq i \leq l\leq k$;
  \vspace{0.1in}
	\item[(ii)]
$
	\nul{i,l}(\xk{k,\tn})=\exp(-c_n (1+o(1)))$, for all  for all $1\leq i \leq l< k$.

\end{itemize}
\end{lemma}

\newcommand{\expi}{e^{\int_0^t \xi_n(X_s)ds}}
\newcommand{\expin}{e^{\int_0^{t_n} \xi_n(X_s)ds}}
\newcommand{\atx}[1]{\mathds{1}\{X_t={#1}\}}
\newcommand{\atxn}[1]{\mathds{1}\{X_{t_n}={#1}\}}
\newcommand{\hitone}[1]{\mathds{1}\{\tau_{\Gamma_{#1}}\leq t\}}
\newcommand{\hittwo}[1]{\mathds{1}\{\tau_{\Gamma_{#1}}> t\}}

\newcommand{\vexp}[1]{\exp\big({\int_0^{#1} V(X_s)ds}\big)}
\newcommand{\vexppp}[1]{\exp\big({\int_0^{#1} [V(X_s)-\lambda_z ]ds}\big)}
\newcommand{\atxtime}[2]{\mathds{1}\{X_{#1}={#2}\}}
\newcommand{\vhitone}{\mathds{1}\{\tau_{\Upsilon}\leq t\}}
\newcommand{\vnohitone}{\mathds{1}\{\tau_{\Upsilon}> t\}}
\newcommand{\vnohittwo}{\mathds{1}\{\tau_{\Lambda}> t\}}
\newcommand{\newgam}{{\Lambda}}

Now we describe a general mechanism that allows us to use these spectral properties to get certain estimates. For the following randomness is not relevant and one can take a general connected graph $\Sigma$ with $\Delta$ denoting the generator of the nearest neighbour simple random walk. Consider a potential $V:\Sigma\to \mathbb{R}$ and subsets $\Upsilon,\newgam\subset \Sigma$ such that $\Upsilon\cap \newgam=\emptyset$. Let $\lambda_z$ be the principle eigenvalue of the operator $\kappa \Delta + V$ on $\big(\Sigma\setminus(\Upsilon\cup \newgam)\big)\cup \{z\}$ zero boundary conditions (this is same as setting $V$ to $-\infty$ on $(\Upsilon\cup \newgam)\setminus \{z\}$). For $z\in \Upsilon$, let $\nu_z$ be the corresponding (positive) eigenfunction normalized so that $\nu_z(z)=1$. Then $\nu_z$ has the probabilistic representation
\begin{equation}\label{evecrep}
\nu_z(x)=\E_x[\vexppp{\tau_z}\mathds{1}
\{\tau_{z}=\tau_{\Upsilon}<\tau_{\newgam}\}].
\end{equation}
Define
\begin{equation}\label{defnw}
\omega(t,x,y)=\E_x \left[\vexp{t}\atxtime{t}{y}\vhitone\vnohittwo\right].
\end{equation}

\begin{lemma}\label{lemmaspectbd}
	For any $t>0$
	\begin{equation}\label{eqnlemmaspectbd}
	\omega(t,x,y)\leq \sum_{z\in\Upsilon}\omega(t,z,y)\nu_z(x)\norm{\nu_z}_2^2.
	\end{equation}
\end{lemma}
This lemma is a version of Theorem 4.1 in \cite{GKM07} but since the results in \cite{GKM07} are written for $\mathbb{Z}^d$ for the sake of completeness we give a proof.
\begin{proof}
	We claim that for any $u\in[0,t]$ and $z\in \Upsilon$
	\begin{equation}\label{claim1}
	\E_z\left[\vexp{t-u}\atxtime{t-u}{y}\incl{\tau_{\newgam}>t-u}\right]\leq e^{-\lambda_z u} \norm{\nu_z}_2^2 \;\omega(t,z,y).
	\end{equation}
We have the following lower bound for $w(t,z,y)$
\begin{equation}\label{claimara}
\begin{aligned}
\omega(t,z,y)&=\E_z\left[\vexp{t}\atxtime{t}{y}\vnohittwo\right]\\
&\geq \E_z\left[\vexp{u}\atxtime{u}{z}\incl{\tau_{\newgam}>u}\incl{\tau_{\Upsilon\setminus\{z\} }>u}\right]
\\&\;\;\;\times \E_z\left[\vexp{t-u}\atxtime{t-u}{y}\incl{\tau_{\newgam}>t-u}\right].
\end{aligned}
\end{equation}
In the first equality above we used the fact that $z\in\Upsilon$.
By the spectral decomposition 
\begin{equation}\label{speclow}
\E_z\left[\vexp{u}\atxtime{u}{z}\incl{\tau_{\newgam}>u}\incl{\tau_{\Upsilon\setminus\{z\} }>u}\right]\geq e^{u\lambda_z}\norm{\nu_z}_2^{-2},
\end{equation}
which implies (\ref{claim1}) through (\ref{claimara}). Since we use the type of estimate in (\ref{speclow}) throughout the rest of this section, here we explain it in detail. Let
\begin{equation}
h(t,x):=\E_x\left[\vexp{t}\atxtime{t}{z}\incl{\tau_{\newgam}>t}\incl{\tau_{\Upsilon\setminus\{z\} }>t}\right].
\end{equation}
Then $h$ solves the parabolic equation
\begin{equation}
\frac{\partial}{\partial t} h(t,x)=\kappa\Delta_n h(t,x) + \xi_n(x) h(t,x) ,\quad t\geq 0,\; x\in\big(\Sigma\setminus(\Upsilon\cup \newgam)\big)\cup \{z\}
\end{equation}
with initial condition $h(0,\cdot)=\delta_z(\cdot)$ and boundary conditions
\begin{equation}
h(t,x)=0,\quad \forall t\geq 0, \; x\in \Upsilon\cup \newgam\setminus\{z\}.
\end{equation}
Therefore, $h$ can be given using the spectrum of $\kappa\Delta+\xi$ with zero boundary conditions on $\Upsilon\cup \newgam\setminus\{z\}$. We already defined $\lambda_z$ and $\nu_z$ as the principle eigenvalue and eigenvector, respectively. Let $\lambda^{(i)}$ and $\nu^{(i)}$, $i=2,\dots, 2^n-|\Lambda|-|\Upsilon|+1:=L$, denote the the rest of eigenvalues and the corresponding eigenvectors, respectively, in the spectrum. Here, eigenvectors have the usual $l_2$ normalization: $\norm{\nu^{(i)}}_2^2=1$. Then, with the initial condition $h(0,\cdot)=\delta_z(\cdot)$, we get
\begin{equation}
h(t,x)=e^{\lambda_z t}\frac{\nu_z(x)\nu_z(z)}{\norm{\nu_z}^2}+\sum_{i=2}^{L}e^{\lambda^{(i)}t} \nu^{(i)}(x)\nu^{(i)}(z).
\end{equation}
For $x=z$ all the coefficients in the second sum becomes $(\nu^{(i)}(z))^2$, thus, non-negative. Since we chose $\nu_z(z)=1$, we arrive at that $h(t,x)\geq e^{t\lambda_z}\norm{\nu_z}_2^{-2}$.
Then, (\ref{claim1}) follows. 

Now we continue with the proof of (\ref{eqnlemmaspectbd}). By the definition (\ref{defnw}) we have
\begin{equation}
\begin{aligned}
\omega(t,x,y)=\sum_{z\in\Upsilon}\E_x\left[\right.&
\vexp{\tau_{\Upsilon}}\atxtime{\tau_\Upsilon}{z}\vhitone\vnohittwo
\\&\left.\E_z\left[\vexp{t-\tau_{\Upsilon}}\atxtime{t-\tau_\Upsilon}{y}\incl{\tau_{\newgam}>t-\tau_\Upsilon}\right]
\right]
\end{aligned}
\end{equation}
Since on $\atxtime{\tau_\Upsilon}{z}$ we have $\tau_{\Upsilon}=\tau_z$, we can replace $\incl{X_{\tau_\Upsilon}=z}$ by $\incl{\tau_z=\tau_{\Upsilon}}$, and $\tau_\Upsilon \leq t < \tau_{\newgam}$ implies that $\tau_\Upsilon<\tau_{\newgam}$ and we get a upper bound if we replace $\vhitone\vnohittwo$ by $\incl{\tau_{\Upsilon}<\tau_{\newgam}}$. Hence,
\begin{equation}
\begin{aligned}
w(t,x,y)&\leq \sum_{z\in \Upsilon} \E_x[\vexp{\tau_z}\mathds{1}
\{\tau_{z}=\tau_{\Upsilon}<\tau_{\newgam}\}
\\&\quad\quad\quad\quad\times \E_z[\vexp{t-\tau_{z}}\atxtime{t-\tau_z}{y}\incl{\tau_{\newgam}>t-\tau_z}]]
\\&\leq \sum_{z\in \Upsilon}\E_x[\vexp{\tau_z}\mathds{1}
\{\tau_{z}=\tau_{\Upsilon}<\tau_{\newgam}\}e^{-\lambda_z \tau_z}]\; \norm{\nu_z}_2^2\;\omega(t,z,y)
\\&=\sum_{z\in\Upsilon} \omega(t,z,y)\norm{\nu_z}_2^2\; \nu_z(x).
\end{aligned}
\end{equation}
For the inequality on the second line we used (\ref{claim1}) and for the equality on the third line we used the representation of $\nu_z$ given in (\ref{evecrep}).
\end{proof}

We divide the expectation in the Feynman-Kac representation (\ref{feynmankac}) of $v_n(t,x,x_{k,\tn})$ into two parts: expectation along the paths that visit $\Gamma_{k-1}$ and those that do not. Namely, 
\begin{equation}
v(t,x,\xk{k,\tn})=\omega_k(t,x,x_{k,\tn})+\tilde{\omega}_k(t,x,x_{k,\tn})
\end{equation}
where
\begin{equation}\label{defnwk}
\omega_{k}(t,x,y):=\E_x\left[\expi \atx{y}\hitone{k-1}\right],
\end{equation}
and
\begin{equation}
\tilde{\omega}_k(t,x,y):=\E_x\left[\expi \atx{y}\hittwo{k-1}\right].
\end{equation}
We first prove the following.
\begin{lemma}\label{lemgrwothkk}
	$P$-a.s. for any $k\in\N$ as $n\to\infty$
	\begin{equation}
	\tilde{\omega}_k(t_n,\xk{k,\tn},\xk{k,\tn})\sim e^{\lambda_{k}t_n}
	\end{equation}
\end{lemma}
\begin{proof}[Proof of Lemma \ref{lemgrwothkk}]
	Using the spectral decomposition of the operator $\kappa \Delta_n +\xi_n$ with zero boundary conditions on $\Gamma_{k-1}$, as discussed before, and part (i) of Lemma \ref{lemmaAna} we get 
	\begin{equation}
	\tilde{\omega}(t_n,x_{k,\tn},x_{k,\tn})\geq e^{\lambda_{k}t_n}\norm{\nu_k}_2^{-2}\sim e^{\lambda_{k}t_n}
	\end{equation}
	We use the spectral gap to get the upper bound
	\begin{equation}
	\tilde{\omega}_k(t_n,x_{k,\tn},x_{k,\tn})\leq e^{\lambda_{k}t_n}\norm{\nu_k}_2^{-2}+e^{\lambda_{k}t_n}e^{-g_{k,k}t_n}\norm{\delta_{x_k}}_2^2.
	\end{equation}
Since $\norm{\delta_{x_{k,\tn}}}_2=1$, using Lemma \ref{lemma1} part (iii), Lemma \ref{lemmaAna2} and Lemma \ref{lemmaAna} part (ii) we are finished with the proof.	
\end{proof}
Note that for $k=1$, $\tilde{\omega}_1(t,x,y)=v_n(t,x,y)$. Hence, the above lemma gives 
\begin{equation}\label{grw1}
v_n(t_n,x_{1,\tn},x_{1,\tn})\sim e^{\lambda_{1}t_n}.
\end{equation}

\newcommand{\cn}{\exp\Big(- c_n (1+o(1))\Big)}
\newcommand{\cnn}{\exp\Big(-2c_n(1+o(1))\Big)}
We need the following result. Recall that we have defined $c_n=\frac{1}{2}n\log n$.
\begin{lemma}\label{lemkiy}
	$P$-a.s. for any $k\in\N$ and for any $i=1,\dots,k-1$ as $n\to\infty$
	\begin{equation}\label{kiy}
 v_n(t_n,x_{i,\tn},x_{k,\tn})\leq e^{\lambda_1 t_n}\cnn +e^{\lambda_i t_n}\cn.
	\end{equation}
\end{lemma}
\begin{proof}[Proof of Lemma \ref{lemkiy}]
	We prove (\ref{kiy}) using strong induction. For $k=1$ there is nothing to check. Now assume that (\ref{kiy}) is true for $1,\dots,k-1$.  Let $i\in\{2,\dots,k-1\}$. We use Lemma \ref{lemmaspectbd} with $\Upsilon=\Gamma_{i-1}$ and $\Lambda=\emptyset$. In this case the corresponding $\omega$ defined in (\ref{defnw}) coincides with $\omega_i$ defined as in (\ref{defnwk}). Note that we have $\omega_i(t,x,y)=v_n(t,x,y)$ if $x\in \Gamma_{i-1}$. Hence, using Lemma \ref{lemmaspectbd}
	\begin{equation}
	w_i(t_n,x_{i,\tn},x_{k,\tn})=w_i(t_n,x_{k,\tn},x_{i,\tn})\leq \sum_{j=1}^{i-1}v_n(t_n,x_{j,\tn},x_{i,\tn})\nu_{j,i-1}(x_{k,\tn})\norm{\nu_{j,i-1}}_2^2
	\end{equation}
	By part (ii) of Lemma \ref{lemmaAna}, since $i<k$, we have $\nu_{j,i-1}(x_{k,\tn})=\cn$ for $j=1,\dots,i-1$ and by part (i) of the same Lemma we have $\norm{\nu_{j,i-1}}_2^2\sim 1$. By the strong induction step we have $v(t_n,x_{j,\tn},x_{i,\tn})\leq e^{\lambda_1 t_n} \cn$. Hence,
	\begin{equation}
     w_i(t_n,x_{i,\tn},x_{k,\tn})\leq e^{\lambda_1 t_n} \cnn.
	\end{equation}
	Now we use Lemma \ref{lemmaspectbd} with $\Upsilon=\{x_{i,\tn}\}$ and $\Lambda=\Gamma_{i-1}$. Since $\incl{X_t=x_{i,\tn}}\incl{\tau_{x_{i,\tn}}\leq t}=\incl{X_t=x_{i,\tn}}$, $\omega(t,\cdot,x_{i,\tn})$ defined in (\ref{defnw}) coincides with $\tilde{\omega}_i(t,\cdot,x_{i,\tn})$, and using Lemma \ref{lemmaspectbd} we get
	\begin{equation}\label{kiymama}
	\tilde{\omega}_i(t_n,x,x_{i,\tn})\leq \tilde{\omega}_i(t_n,x_{i,\tn},x_{i,\tn})\nu_i(x)\norm{\nu_i}_2^2.
	\end{equation}
	Hence, by parts (i) and (ii) of Lemma \ref{lemmaAna} and Lemma \ref{lemgrwothkk} we have
	\begin{equation}
	\tilde{\omega}_i(t_n,x_{k,\tn},x_{i,\tn})\leq e^{\lambda_i t_n} \cn.
	\end{equation}
	Since $v_n=\omega_i+\tilde{\omega}_i$, we have proved that (\ref{kiy}) holds true for $i=2,\dots,k-1$. For $i=1$ recall that $v_n=\tilde\omega_1$. Similar to how we arrived at (\ref{kiymama}) we get
	\begin{equation}\label{inqv1}
	v_n(t_n,x,x_{1,\tn})\leq v_n(t_n,x_{1,\tn},x_{1,\tn})\nu_1(x)\norm{\nu_1}_2^2.
	\end{equation}
	Hence, using once again part (ii) of Lemma \ref{lemmaAna} and Lemma \ref{lemgrwothkk} we have 
	\begin{equation}
v_n(t_n,x_{1,\tn},x_{k,\tn})\leq e^{\lambda_1 t_n}\cn
	\end{equation}
	This completes the proof.

\end{proof}

\begin{lemma}\label{lemvk}
	$P$-a.s. for any $k\in \N$ as $n\to\infty$
	\begin{equation}
	v_n(t_n,x_{k,\tn},x_{k,\tn})\leq e^{\lambda_1 t_n}\cnn +e^{\lambda_k t_n}.
	\end{equation}
\end{lemma}
\begin{proof}[Proof of Lemma \ref{lemvk}]
	Using Lemma \ref{lemmaspectbd} with $\Gamma=\Gamma_{k-1}$ and $\newgam=\emptyset$ we get
	\begin{equation}
	\omega_k(t_n,x_{k,\tn},x_{k,\tn})\leq \sum_{i=1}^{k-1} v_n(t_n,x_{i,\tn},x_{k,\tn}) \nu_{i,k-1}(x_{k,\tn}).
	\end{equation}
	Hence, using part (ii) of Lemma \ref{lemmaAna} and Lemma \ref{lemkiy} we have
	\begin{equation}
	\omega_k(t_n,x_{k,\tn},x_{k,\tn})\leq e^{\lambda_1 t_n} \cnn.
	\end{equation}
	Finally, Lemma \ref{lemgrwothkk} finishes the proof.
\end{proof}

\begin{lemma}\label{oldschool}
	$P$-a.s. for any $k\in \N$ as $n\to\infty$
	\begin{equation}
	v_n(t_n,x_{1,\tn},x_{k,\tn})\geq e^{\lambda_1 t_n}\cn.
	\end{equation}
\end{lemma}
\begin{proof}[Proof of Lemma \ref{oldschool}]
	Note that the description in (\ref{repeigvec}) gives
	\begin{equation}
	\nu_1(x)=\E_x\left[\exp\Big(\int_0^{\tau_{\xk{1,\tn}} } \big[\xi_n(X_s)-\lamm{1}\big]ds  \Big)\right].
	\end{equation}
	Hence, using the Feynman-Kac formula, the spectral decomposition of $\kappa \Delta_n+\xi_n$ we get
	\begin{equation}
	\begin{aligned}
	v_n(t_n,x,x_{1,\tn})&=\E_x\left[\expin \atxn{x_{1,\tn}}\right]
	\\&=\E_x\left[e^{\int_0^{\tau_{x_{1,\tn}}}\xi_n(X_s)ds } E_{x_{1,\tn}}[e^{\int_0^{t_n-\tau_{x_{1,\tn}}}\xi_n(X_s)ds} \incl{X_{t_n-\tau_{x_{1,\tn}}}=x_{1,\tn}}]\right]
	\\&\geq \E_x\left[e^{\int_0^{\tau_{x_{1,\tn}}}\xi_n(X_s)ds } e^{\lambda_1 (t_n-\tau_{x_{1,\tn}})}\norm{\nu_1}_2^{-2}\right]
	\\&=e^{t_n\lambda_1}\norm{\nu_1}_2^{-2} \E_x\left[e^{\int_0^{\tau_{x_{1,\tn}}}\big(\xi_n(X_s)-\lambda_1\big)ds}\right]=e^{t_n\lambda_1} \norm{\nu_1}_2^{-2}\nu_1(x).
	\end{aligned}
	\end{equation}
	Then, part (i) and (ii) of Lemma \ref{lemmaAna} finish the proof.
\end{proof}
Note that, Lemma \ref{lemkiy} and Lemma \ref{oldschool} yield that $P$-a.s. for any $k\in\N\setminus\{1\}$ as $n\to\infty$
\begin{equation}\label{growth1k}
v_n(t_n,x_{1,\tn},x_{k,\tn})=e^{\lambda_1 t_n}\cn
\end{equation}
Now we are ready to prove the main results.
\begin{proof}[Proof of Theorem \ref{thrmLoc} and Theorem \ref{thrmGrow}]
	We first prove the statements for $k=1$, namely, (\ref{thrmLocx1}) and (\ref{thrmGrowx1}). Recalling (\ref{inqv1}) and that $v_n(t,x_{1,\tn})=\sum_{x\in\Sigma_n}v(t,x,x_{1,\tn})$ we have
	\begin{equation}
	\sum_{x\not= x_{1,\tn}} v_n(t,x,x_{1,\tn})\leq v_n(t,x_{1,\tn},x_{1,\tn})\sum_{x\not= x_{1,\tn}}\nu_1(x)
	\end{equation}
	Hence, using part (i) of Lemma \ref{lemmaAna} as $n\to\infty$ we have $v_n(t_n,x_{1,\tn})\sim v_n(t_n,x_{1,\tn},x_{1,\tn})$, that is, (\ref{thrmLocx1}). Finally, (\ref{grw1}) finishes the proof of (\ref{thrmGrowx1}).
	Now we assume $k\in\N\setminus{\{1\}}$. Using Lemma \ref{lemmaspectbd} with $\Upsilon=\Gamma_{k-1}$ and $\Lambda=\emptyset$ and Lemma \ref{lemkiy} we have as $n\to\infty$
	\begin{equation}\label{son1}
	\begin{aligned}
	\omega_k(t_n,x,x_{k,\tn})&\leq v_n(t_n,x_{1,\tn},x_{k,\tn}) \nu_{1,k-1}(x)+\sum_{i=2}^{k-1} v_n(t_n,x_{i,\tn},x_{k,\tn})\nu_{i,k-1}(x)
	\\&\leq v_n(t_n,x_{1,\tn},x_{k,\tn}) \nu_{1,k-1}(x)+\\&\left[e^{\lambda_1 t_n}\cnn + e^{\lambda_2 t_n}\cn\right]\sum_{i=2}^{k-1} \nu_{i,k-1}(x).
	\end{aligned}
	\end{equation}
    The same reasoning we used to get (\ref{kiymama}) yields
    \begin{equation}
    \tilde{\omega}_k(t_n,x,x_{k,\tn})\leq \tilde{\omega}_k(t_n,x_{k,\tn},x_{k,\tn}) \nu_k(x)\norm{\nu_k}_2^2.
    \end{equation}
	
	We separate the short and long time regimes.
	
\underline{Short time regime}:
The key point is that in this time regime, by Lemma \ref{lemmaAna2},
\begin{equation}\label{timeregime}
e^{\lambda_1 t_n}\cn\ll e^{\lambda_k t_n}. 
\end{equation}
Recall that $v_n=\omega_k+\tilde{\omega}_k$. Due to (\ref{timeregime}), Lemma \ref{lemgrwothkk} and Lemma \ref{lemvk} yield
\begin{equation}
v_n(t_n,x_{k,\tn},x_{k,\tn})\sim e^{\lambda_k t_n}.
\end{equation}
For the second item on the right hand side of the last inequality in (\ref{son1}), (\ref{timeregime}) gives
\begin{equation}
\left[e^{\lambda_1 t_n}\cnn + e^{\lambda_2 t_n}\cn\right]\ll e^{\lambda_k t_n}.
\end{equation}
By part (i) of Lemma \ref{lemmaAna} we have $\sum_x\sum_{i=2}^{k-1} \nu_{i,k-1}(x)\leq k(1+o(1))$. Via Lemma \ref{lemkiy} and (\ref{timeregime}) we get $v_n(t_n,x_{1,\tn},x_{k,\tn})\ll e^{\lambda_k t_n}$. Hence, applying again part (i) and (ii) of Lemma \ref{lemmaAna} we have
\begin{equation}
\sum_{x}\omega_k(t_n,x,x_k)\ll v_n(t_n,x_k,x_k)
\end{equation}
 Hence, by part (i) of Lemma \ref{lemmaAna}
\begin{equation}
{\sum_{x\not= x_{k,\tn}} \tilde{\omega}_k(t_n,x,x_{k,\tn})}\leq\tilde{\omega}_k(t_n,x_{k,\tn},x_{k,\tn}) \sum_{x\not= x_{k,\tn}}\nu_k(x)\norm{\nu_k}_2^2\ll \tilde{\omega}_k(t_n,x_{k,\tn},x_{k,\tn})
\end{equation}
and we get $v_n(t_n,x_{k,\tn})\sim \tilde{\omega}_k(t_n,x_{k,\tn},x_{k,\tn})$. Since $\tilde{\omega}_k(t,x_k,x_k)\leq v_n(t,x_k,x_k)$ we reach at
\begin{equation}
u(t_n,x_{k,\tn},x_{k,\tn})=\frac{v_n(t_n,x_{k,\tn},x_{k,\tn})}{v_n(t_n,x_{k,\tn})}\longrightarrow 1.
\end{equation}
This finishes the proof of the statement in Theorem \ref{thrmLoc} concerning short time scales. By Lemma \ref{lemgrwothkk} we have $\tilde{\omega}_k(t_n,x_{k,\tn},x_{k,\tn})\sim e^{\lambda_k t_n}$ and by Lemma \ref{lemmaAna2} $\lambda_k=\xi_{k}-\kappa +\Theta(1/n^2)$. Since on short time scales $t_n\ll n^2$ we have $e^{\lambda_{k,\tn}\tn}\sim e^{(\xi_{k,\tn}-\kappa)t_n}$. Hence, we are finished with the proof Theorem \ref{thrmGrow} for short time scales.

\underline{Long time regime}: In this time regime, by Lemma \ref{lemmaAna2},
\begin{equation}
e^{\lambda_1 t_n}\cn\gg e^{\lambda_k t_n}.
\end{equation}
By Lemma \ref{lemmaAna2} and part (ii) of Lemma \ref{lemma1} there exist random positive constants $C$ and $C'$ such that $\lambda_1-\lambda_2>C$ and $\lambda_1-\lambda_k>C'$. By the latter and Lemma \ref{lemmaAna2}, in this regime we have $t_n> C'' n\log n$ for some random positive constant $C''$. Therefore, for some $\e_n\to 0$,
\begin{equation}
\frac{e^{\lambda_2 t_n} \cn}{e^{\lambda_1 t_n}\cn}\leq e^{-C t_n}e^{\e_n n \log n}\leq e^{-(C''n-\e_n)n\log n}.
\end{equation}
By part (i) of Lemma \ref{lemmaAna} and H\"older's inequality
\begin{equation}
\sum_{x\in\Sigma_n}\sum_{i=2}^{k-1} \nu_{i,k-1}(x)\leq 2k 2^{n/2}. 
\end{equation}
Then, since $n\ll n\log n$, using (\ref{growth1k}) we get
\begin{equation}
\sum_{x}\left(e^{\lambda_1 t_n}\cnn + e^{\lambda_2 t_n}\cn\right)\sum_{i=2}^{k-1} \nu_{i,k-1}(x)\ll v_n(t_n,x_{1,\tn},x_{k,\tn}).
\end{equation}
Using part (i) of Lemma \ref{lemmaAna} we have
\begin{equation}
\sum_{x\not= x_{1,\tn}}v_n(t_n,x_{1,\tn},x_{k,\tn})\nu_{1,k-1}(x)\ll v_n(t_n,x_{1,\tn},x_{k,\tn}),
\end{equation}
and conclude through (\ref{son1}) that $\sum_{x}\omega_k(t_n,x,x_{k,\tn})\sim v_n(t_n,x_{1,\tn},x_{k,\tn})$. Once more using part (i) of Lemma \ref{lemmaAna}, Lemma \ref{lemgrwothkk} and Lemma \ref{oldschool} we have
\begin{equation}
   \sum_x \tilde{\omega}_k(t_n,x,x_{k,\tn})\leq \tilde{\omega}_k(t_n,x_{k,\tn},x_{k,\tn})\sum_{x} \nu_k(x)\sim \tilde{\omega}_k(t_n,x_{k,\tn},x_{k,\tn})\sim e^{\lambda_k t_n}\ll v_n(t_n,x_{1,\tn},x_{k,\tn}).
\end{equation}
Hence, $v_n(t_n,x_{k,\tn})\sim v_n(t_n,x_{1,\tn},x_{k,\tn})$, and (\ref{growth1k}) finishes the proof for the long time scales. 
\end{proof}

\section{Proof of spectral results}\label{sec3}

In this section we prove  of Lemma \ref{lemmaAna2} and Lemma \ref{lemmaAna}. For proving the results about eigenvalues we first give a general result for a given potential on $\Sigma_n$ which is similar in spirit to the cluster expansion result given in \cite{GM98} (Lemma 2.18, on page 45). Let $V:\Sigma_n\to[-\infty,\infty)$ be a potential and $A\subset \Sigma_n$ be such that
\begin{equation}
d_{\min}(A)=\min\{ d(x,y):x,y\in A,x\not= y\}>2.
\end{equation}
We set
\begin{equation}
N:=\max_{A} V,\quad   M:=\max_{\Sigma_n\setminus A} V.
\end{equation}
\begin{lemma}\label{lemmasimpleeigen}
If
\begin{equation}
M\leq N-\kappa,
\end{equation}
then 
\begin{equation}\label{inqeigv}
N-\kappa\leq\lambda_1 <\gamma
\end{equation}
for any $\gamma>N-\kappa$ with
\begin{equation}\label{inqnew}
\frac{\kappa}{\gamma-(N-\kappa)}<\frac{n(\gamma-M)}{\kappa}.
\end{equation}

\end{lemma}
\begin{proof}[Proof of Lemma \ref{lemmasimpleeigen}]
The lower bound in (\ref{inqeigv}) follows easily by replacing $V$ with $-\infty$ everywhere expect at the site in $A$ where the maximum value $N$ is reached and using the fact that $\lambda_1$ is non-decreasing in $V$.

For the upper bound we will show that any $\gamma>N-k$ that satisfies (\ref{inqnew}) is in the resolvent set of $\kappa\Delta_n+V$. This is enough because if $\gamma>0$ satisfies (\ref{inqnew}) then so does any $\gamma'>\gamma$. We denote by $\mathcal{R}_\gamma$ the resolvent at $\gamma$. Using the probabilistic representation of the resolvent, since we are on a finite space, it is enough to check that

\begin{equation}\label{resolve}
\mathcal{R}_\gamma\mathbf{1}(x)=\E_x[\int_0^\infty dt\;  e^{\int_0^t (V(X_s)-\gamma)ds}]<\infty
\end{equation}
for any $x\in \Sigma_n$ where $\mathbf{1}$ denotes the constant function of 1. We define hitting times $0\leq \sigma_0<\tau_0<\sigma_1<\tau_1<\cdots$ by
\begin{equation}
\sigma_0 =\inf\{t\geq 0:\; X(t)\in A\},
\end{equation}
and for $i\in\N\cup\{0\}$
\begin{equation}
\begin{aligned}
\tau_i&=\inf\{t\geq \sigma_i:\; X(t)\notin A\},
\\
\sigma_{i+1}&=\inf\{t\geq \tau_i:\; X(t)\in A\}.
\end{aligned}
\end{equation}
Note that, since $d_{\min}(A)>2$, $\tau_i-\sigma_i$, $i=0,1,\dots$, is an independent sequence of exponential random variables with rate $\kappa$. Using these stopping times we can write
\begin{equation}
\mathcal{R}_\gamma\mathbf{1}(x)=\E_x[\int_0^{\sigma_0} dt\;  e^{\int_0^t (V(X_s)-\gamma)ds}]+\sum_{i=0}^\infty \E_x[\int_{\sigma_i}^{\sigma_{i+1}} dt\;  e^{\int_0^t (V(X_s)-\gamma)ds}].
\end{equation}
Note that $V(x)\leq M \leq N-\kappa<\gamma$ for $x\notin A$. Since $X(t)\notin A$ for $t\in[0,\sigma_0)$
\begin{equation}
\E_x[\int_0^{\sigma_0} dt\;  e^{\int_0^t (V(X_s)-\gamma)ds}]\leq \int_0^\infty dt\;  e^{\int_0^t (M-\gamma)ds}=\int_0^\infty dt\;  e^{-t(\gamma-M)}=\frac{1}{\gamma-M}<\infty.
\end{equation}
For $i=0,1,\dots$ we have
\begin{equation}
\begin{aligned}
\E_x[\int_{\sigma_i}^{\sigma_{i+1}} dt\;  e^{\int_0^t (V(X_s)-\gamma)ds}]&=\E_x\Big[e^{\int_0^{\sigma_0} (V(X_s)-\gamma)ds }\\&\hspace{1cm}\times  \E_{X(\sigma_0)}[e^{\int_0^{\sigma_i}(V(X_s)-\gamma)ds}  ]
\\&\hspace{1cm}\times 
\E_{X(\sigma_i)}[\int_{0}^{\sigma_{1}} dt\;  e^{\int_0^t (V(X_s)-\gamma)ds}]\Big].
\end{aligned}
\end{equation}
Since $V(X_s)<\gamma$ for $s\in[0,\sigma_0)$ we have $e^{\int_0^{\sigma_0} (V(X_s)-\gamma)ds}<1$. By the strong Markov property
\begin{equation}
\E_{X(\sigma_0)}[e^{\int_0^{\sigma_i}(V(X_s)-\gamma)ds}  ]\leq \left(\max_{x\in A} \E_x[e^{\int_0^{\sigma_1} (V(X_s)-\gamma)ds }]\right)^i.
\end{equation}
Hence,
\begin{equation}
\begin{aligned}
\E_x[\int_{\sigma_i}^{\sigma_{i+1}} dt\;  e^{\int_0^t (V(X_s)-\gamma)ds}]&\leq \left(\max_{x\in A} \E_x[e^{\int_0^{\sigma_1} (V(X_s)-\gamma)ds }]\right)^i
\\&\;\;\; \times \max_{x\in A}\E_x\left[\int_0^{\sigma_1}dt\;  e^{\int_0^t (V(X_s)-\gamma)ds}\right].
\end{aligned}
\end{equation}
Therefore, to finish the proof of (\ref{resolve}) it is enough to check that
\begin{equation}\label{check1}
\max_{x\in A} \E_x[e^{\int_0^{\sigma_1} (V(X_s)-\gamma)ds }]<1
\end{equation}
and
\begin{equation}\label{check2}
\max_{x\in A}\E_x[\int_0^{\sigma_1}dt\;  e^{\int_0^t (V(X_s)-\gamma)ds}]<\infty.
\end{equation}
For the former we write
\begin{equation}
\E_x[e^{\int_0^{\sigma_1} (V(X_s)-\gamma)ds }]=\E_x\Big[e^{\int_0^{\tau_0} (V(X_s)-\gamma)ds }\E_{X(\tau_0)} [e^{\int_0^{\sigma_0}    (V(X_s)-\gamma)ds}]\Big].
\end{equation}
Since $d_{\min}(A)>2$ any $z\notin A$ has at most one neighbour that is in $A$. Hence, $\sigma_0$, for the walk starting from any $z\notin A$, is stochastically bounded from below by an exponential random variable with rate $\kappa/n$. Hence, using once again that $V\leq M<\gamma$ on $A^c$, we can conclude that for any $z\notin A$ 
\begin{equation}
\E_z[e^{\int_0^{\sigma_0}    (V(X_s)-\gamma)ds}]< \E_z[e^{-(\gamma-M)\sigma_{0}}]\leq  \frac{\kappa/n}{\kappa/n+\gamma-M}<\frac{\kappa}{n(\gamma-M)}.
\end{equation}
Also, recall that starting from $x\in A$, $\tau_0$ is distributed as an exponential random variable with rate $\kappa$. Hence, 
\begin{equation}\label{estimatesomeeigen}
\begin{aligned}
\max_{x\in A}\E_x[e^{\int_0^{\sigma_1} (V(X_s)-\gamma)ds }]&=\max_{x\in A}\E_x\Big[e^{\int_0^{\tau_0} (V(X_s)-\gamma)ds }\E_{X(\tau_0)} [e^{\int_0^{\sigma_0}    (V(X_s)-\gamma)ds}\Big]
\\&\leq \max_{x\in A}\E_x\Big[ e^{\int_0^{\tau_0} (V(X_s)-\gamma)ds }\Big]\frac{\kappa}{n(\gamma-M)}
\\&=\max_{x\in A}\frac{\kappa}{\kappa+\gamma-V(x)}\frac{\kappa}{n(\gamma-M)}
\\&= \frac{\kappa}{\kappa+\gamma-N}\frac{\kappa}{n(\gamma-M)}.
\end{aligned}
\end{equation}
By (\ref{inqnew}) the last quantity above is  less than 1 and thus, (\ref{check1}) is satisfied. Now it remains to check (\ref{check2}). To this end we write
\begin{equation}\label{check3}
\E_x[\int_0^{\sigma_1}dt\;  e^{\int_0^t (V(X_s)-\gamma)ds}]=\E_x[\big(\int_0^{\tau_0} +\int_{\tau_0}^{\sigma_0}\big)dt\;  e^{\int_0^t (V(X_s)-\gamma)ds}].
\end{equation}
 For any $x\in A$ for a $c>0$ appropriately chosen $V(x)-\gamma\leq N-\gamma<c<\kappa$. Hence, 
\begin{equation}
\E_x[\int_0^{\tau_0}dt\;  e^{\int_0^t (V(X_s)-\gamma)ds}]\leq \E_x[\frac{e^{\tau_0 c}-1}{c}].
\end{equation}
Now since $\tau_0$ is distributed as an exponential random variable with rate $\kappa$ and $c<\kappa$ the above quantity is finite. The second integral on the right hand side of (\ref{check3}) is equal to
\begin{equation}\label{check4}
\E_x[e^{\int_0^{\tau_0}(V(X_s)-\gamma)ds}\E_{X(\tau_0)}[\int_{0}^{\sigma_0}dt\;  e^{\int_0^t (V(X_s)-\gamma)ds}]].
\end{equation}
Since $V\leq M<\gamma$ on $A^c$ for any $z\notin A$
\begin{equation}
\E_{z}[\int_{0}^{\sigma_0}dt\;  e^{\int_0^t (V(X_s)-\gamma)ds}]\leq \int_0^\infty dt e^{-t(\gamma-M)}<\infty.
\end{equation}
We have already seen in (\ref{estimatesomeeigen}) that $\max_{x\in A}\E_x[ e^{\int_0^{\tau_0} (V(X_s)-\gamma)ds }]=\frac{\kappa}{\gamma-(N-\kappa)}$. Thus, (\ref{check4}) is finite and (\ref{check2}) is satisfied. This completes the proof of the lemma.
\end{proof}

The key ingredient of the proofs of Lemma \ref{lemmaAna2}  and Lemma \ref{lemmaAna} is the next result. For $\delta\in(0,1)$ define
\begin{equation}
A_n^{\delta}:=\{x:\eta(x)\geq n \delta	 \log 2 \}.
\end{equation}
Let 
\begin{equation}
I(x):=x\log x+(1-x)\log(1-x)+\log 2,\quad x\in[0,1]
\end{equation}
be Cramer's rate function.
\begin{lemma}\label{lemma2}\
	\begin{itemize}
		\item[(i)]
		Let $\delta>1/2$ and $\omega^\delta$ be the unique solution of $I(\omega^\delta)=2(1-\delta) \log 2$. Then $P$-a.s. for any $c < \omega^\delta$ 
		\begin{equation}
		d_{min}(A^\delta_n):=\min\Big\{d(x,y):x,y\in A_n^{\delta},x\not=y\Big\}\geq c n,
		\end{equation} 
		for all $n$ large enough.
		\item[(ii)] $P$-a.s. for any $i,k\in\N$ with $i\not=k$
		\begin{equation}
		d(x_i,x_k)\sim n/2,
		\end{equation}
				for all $n$ large enough.
	\end{itemize}
\end{lemma}

\begin{proof}[Proof of \ref{lemma2}]
	Since $P(\eta(x)\geq n \delta	 \log 2 )=2^{-\delta n}$ the statement of part (i) is same as part (ii) of Lemma ?? on page ?? of \cite{BAC08} with $\delta$ is replaced by $\gamma$ in the notation used in the aforementioned article. For part $(ii)$ note that for any $i,k$ the distribution of $d(x_i,x_k)$ is that of a Binomial random variable with parameters $n$ and $1/2$, conditioned to be non-zero. Hence, the result follows from strong law of large numbers.
\end{proof}

\begin{proof}[Proof of Lemma \ref{lemmaAna2}]
For $\delta\in(1/2,1)$ let
\begin{equation}
A_n^\delta=\{x:\eta(x)\geq n \delta \log 2 \}
\end{equation}\newcommand{\cd}{c^{\delta}}
Then by Lemma \ref{lemma2} for some $c\in(0,\omega^\delta)$, $P$-a.s. 
\begin{equation}
d_{\min}(A_n^\delta):=\min\{d(x,y):\;x,y\in A_n^\delta,x\not= y\}\geq cn
\end{equation}
for all $n$ large enough. We use Lemma \ref{lemmasimpleeigen} with $V$ given by $V\equiv \xi$ on $\Sigma_n\setminus \Gamma_l\cup\{x_{i,\tn}\}$ and $V\equiv -\infty $ on $\Gamma_l\setminus\{x_{i,\tn}\}$. Part (i) of Lemma \ref{lemma1} and the fact that $f$ in Assumption $(R1)$ is strictly increasing imply that $P$-a.s. $x_{i,\tn}\in A_n^\delta$ for $n$ large enough. This yields
\begin{equation}
N=\max_{A_n^\delta}V=\xi_{i,\tn}.
\end{equation}
Hence, with $\gamma=\xi_{i,\tn}-\kappa+\e_n$,
\begin{equation}
\lambda_{i,l}\leq \xi_{i,\tn}-\kappa+\e_n
\end{equation}
if
\begin{equation}\label{soninq}
\frac{\kappa}{\e_n}<\frac{n(\xi_{i,\tn}-\kappa+\e_n-M)}{\kappa}.
\end{equation}
By the definition of $A_n^\delta$ and Assumption $(R1)$ we have $M\leq \psi_n(n\delta\log 2)\leq \delta_1\theta n$ for some $\delta_1\in(0,1)$. By Lemma \ref{lemma1} $\xi_i\geq \delta_2\theta n$ for some $\delta_2\in(\delta_1,1)$. Thus, (\ref{soninq}) is satisfied if
\begin{equation}
\frac{\kappa}{\e_n}< \frac{(\delta_2-\delta_1)\theta n^2-\kappa n}{\kappa}. 
\end{equation}
Hence, we can choose the  sequence $\e_n$ so that $\e_n=C/n^2$.  
Therefore,
\begin{equation}
\lambda_{i,l}\leq \xi_{i,\tn}-\kappa + O(1/n^2).
\end{equation}
Now we prove the lower bound for $\lambda_{i,l}$. Let $l_n$ be as in Assumption $(L)$. We first claim that $P$-a.s. 
\begin{equation}\label{neighbors}
\xi_n(y)\geq -l_n,\quad \forall y\sim x_{k,\tn},\; \forall k\in \N. 
\end{equation}
Note that the random variables $\xi_n(y),y\sim x_{k,\tn}$, are independent and have the distribution of $\xi_n$ conditioned on not being the $k$-th largest. We have the following obvious bound for the latter
\begin{equation}
P(\xi_n(y)\leq -l_n| \;y\not=x_{k,\tn})=\frac{P(\xi_n(y)\leq -l_n, \;y\not=x_{k,\tn})}{P(y\not= x_{k,\tn})}\leq\frac{G_n(-l_n)}{P(y\not= x_{k,\tn})}= \frac{G_n(-l_n)}{1-1/\tn}.
\end{equation}
Consequently,
\begin{equation}
P(\exists y \sim x_{k,\tn} \mbox{ s.t. }\xi_n(y)\leq -l_n )\leq CnG_n(-l_n).
\end{equation}
By Assumption $(L)$ the last quantity above is summable in $n$, and an application of Borel-Cantelli lemma proves (\ref{neighbors}).  Using (\ref{neighbors}) we have $\lambda_{i,l}$ is bounded below by the principle eigenvalue, $\tilde{\lambda}$, of $\kappa \Delta_n +V$ on $x_{i,\tn}\cup\{y:y\sim x_{i,\tn}\}$ with zero boundary conditions, where $V(x_{i,\tn})=\xi_{i,\tn}$ and $V(y)=-l_n$ for $y\sim x_{i,\tn}$. Since $\xi_{i,\tn}-\kappa \gg -l_n$, the principle eigenvalue of the operator one gets by collapsing the neighbours of $x_{k,\tn}$ to a single state with potential $-l_n$ is same as $\tilde{\lambda}$. The matrix representation of the the states operator is
\begin{equation}
\left[
\begin{array}{ll}
\xi_{i,\tn} -\kappa&\kappa\\
\kappa/n&-\kappa(1-/n){-l_n}
\end{array}
\right].
\end{equation}
Using the fact that $l_n\ll n$ (by Assumption $(L)$), a simple calculation shows that the principle eigenvalue of the above matrix is
\begin{equation}
\xi_i-\kappa +\frac{C\kappa^2}{n\xi_{i,\tn}}+o(n^{-2}).
\end{equation}
Finally, since $\xi_{i,\tn}\sim \theta n$ we have the right upper and conclude that
\begin{equation}
\lambda_{i,l}= \xi_{i,\tn}-\kappa +\Theta(n^{-2}).
\end{equation}

For the spectral gap, using the min-max formula, we have that the second largest eigenvalue is bounded above by the principle eigenvalue of $\kappa \Delta_n +\xi_n$ with zero boundary conditions on $\Gamma_l$. With the same exact proof above we get that this principle eigenvalue is $\xi_{l+1,\tn}-\kappa+o(1)$ (since $N$ in this case is $\xi_{l+1,\tn}$). Hence, we are finished with the proof of the spectral gap.  

\end{proof}

\begin{proof}[Proof of Lemma \ref{lemmaAna} part (i)]
Since $\nu_{i,l}$ is the principle eigenfunction of a symmetric operator, by Perron-Frobenius theorem its values are non-negative. Therefore, recalling that $\nu_{i,l}(x_i)=1$, $\sum_{x\not= x_1}\nu_{i,l}(x)\to 0$ implies $\norm{\nu_{i,l}}_2\to 1$. 

Now we show that $\sum_{x\not= x_1}\nu_{i,l}(x)\to 0$. For $\delta\in(1/2,1)$ let $A_n^\delta$ and $\omega^\delta$ be as in Lemma \ref{lemma2}. We again set $\xi=\infty$ on $\Gamma_l\setminus\{x_i\}$, and define $B_n:=B(x_i,c n-3)$ for some $c\in(0,\omega^\delta)$. We will first prove that
\begin{equation}\label{ara1}
\max_{x\notin B_n} \nu_{i,l}(x)\leq \cn,
\end{equation}	
where as before $c_n=\frac{1}{2}n\log n$. We write
\begin{equation}\label{inshson}
\begin{aligned}
\nu_{i,l}(x)&= \E_{x}[\exp(\int_0^{\tau_{x_{i,\tn}}} (\xi(X_s)-\lambda_{i,l} ) ds) \incl{\tau_{x_{i,\tn}}=\tau_{\Gamma_l}} \incl{\tau_{A_n^\delta \setminus\{x_{i,\tn}\}} >\tau_{x_{i,\tn}}}]
\\&+\E_{x}[\exp(\int_0^{\tau_{x_{i,\tn}}} (\xi(X_s)-\lambda_{i,l} ) ds) \incl{\tau_{x_{i,\tn}}=\tau_{\Gamma_l}} \incl{\tau_{A_n^\delta \setminus\{x_{i,\tn}\}} \leq \tau_{x_{i,\tn}}  }]
\end{aligned}
\end{equation}
Since $X(s)\notin A_n^\delta$ for any $s\in [0,\tau_{x_{i,\tn}})$ on the event $\tau_{A_n^\delta\setminus\{x_{i,\tn}\}}>\tau_{x_{i,\tn}}$, using the definition of $A_n^\delta$ and Assumption $(R1)$ we have $\xi(X_s)\leq \theta'n$ on the same event, for some $\theta'<\theta$. Hence, the first expectation on the right hand side of (\ref{inshson}) is bounded above by
\begin{equation}
\E_x[\exp(\int_0^{\tau_{x_{i,\tn}}} (\theta' n-\lambda_{i,l}) ds) ]
\end{equation}
By Lemma \ref{lemmaAna2} we have $\lambda_{i,l}=\xi_{i,\tn}-\kappa +o(1)$. As a result, via Lemma \ref{lemma1}, $\lambda_{i,l}\gg\theta' n$. Finally, since $x\notin B_n$, $\tau_{x_{i,\tn}}$ is stochastically bounded below by the sum of $cn-3$ i.i.d. exponentials with rate $\kappa$, and this yields
\begin{equation}
\E_x[\exp(\int_0^{\tau_{x_{i,\tn}}} (\theta' n-\lambda_{i,l}) ds) ]\leq \left[\frac{\kappa}{\kappa+\lambda_{i,l}-\theta' n}\right]^{c n-3}.
\end{equation}
Using the fact that $\lambda_{i,l}=\xi_{i,\tn}-\kappa +o(1)$ and Lemma \ref{lemma1} we get
\begin{equation}\label{inshson2}
\left[\frac{\kappa}{\kappa+\lambda_{i,l}-\theta' n}\right]^{c n-3}=\exp\left(-c n\log n (1+o(1))\right).
\end{equation}
On the event $\incl{\tau_{A_n^\delta \setminus\{x_{i,\tn}\}} \leq \tau_{x_{i,\tn}}  }$ we have $\xi(X_s)\leq \theta' n\ll \lambda_{i,l}$ for any $s\in [0,\tau_{A_n^\delta \setminus\{x_{i,\tn}\}})$. Hence, using the strong Markov property the second term on the right hand side of (\ref{inshson}) is bounded above by
\begin{equation}
\max_{x\in A_n^\delta\setminus\{x_i\}} \nu_{i,l}(x).
\end{equation}  
Since $d_{min}(A_n^\delta)\geq c n$ and $x_i\in A_n^\delta$ for $n$ large enough, if $x\in A_n^\delta\setminus\{x_i\}$ then for any $y\sim x$ we have $y\notin A_n^\delta$ and $d(y, B_n)>1$. Hence,
\begin{equation}
\begin{aligned}
\max_{x\in A_n^\delta\setminus\{x_{i,\tn}\}} \nu_{i,l}(x)&\leq \frac{\kappa}{\kappa+\lambda_{i,l}-\xi_{l+1,\tn}}\frac{\kappa}{\kappa+\lambda_{i,l}-\theta' n}\max_{x\notin B_n}\nu_{i,l}(x)
\\&\leq\frac{C}{n}\max_{x\notin B_n}\nu_{i,l}(x)
\end{aligned}
\end{equation} 
for some positive constant $C$. For the first inequality in the above display we used the fact that $\xi(x)\leq \xi_{l+1,\tn}$ for $ x \in A_n^\delta\setminus\{x_{i,\tn}\}$ and $\xi(x)\leq \theta'n$ for $x\notin A_n^\delta$, and for the second inequality we used both parts of Lemma \ref{lemma1} and Lemma \ref{lemmaAna2}. 
Hence, together with (\ref{inshson}) and (\ref{inshson2}) we get
\begin{equation}
\max_{x\notin B_n}\nu_{i,l}(x)\leq \exp\left(-c n\log n (1+o(1))\right)+\frac{C}{n}\max_{x\notin B_n}\nu_{i,l}(x).
\end{equation}
Then
\begin{equation}
\begin{aligned}
\max_{x\notin B_n}\nu_{i,l}(x)&\leq (1-C/n)^{-1}\exp\left(-c n\log n (1+o(1))\right)
\\&=\exp\left(-c n\log n (1+o(1))\right).
\end{aligned}
\end{equation}
Since $\omega^\delta\to 1/2$ as $\delta\to 1$, by part (i) of Lemma \ref{lemma2}, as $\delta\to 1$, we can choose $c$ as close to $1/2$ as we wish. Hence, we are finished with the proof of (\ref{ara1}).
By (\ref{ara1}) we have
\begin{equation}
\sum_{x\in B_n^c}\nu_{i,l}(x)\leq 2^n \exp\left(-\frac{1}{2} n\log n (1+o(1))\right)\longrightarrow 0.
\end{equation}
Hence, it remains to prove that 
\begin{equation}
\sum_{x\in B_n\setminus\{x_i\}}\nu_{i,l}(x)\to 0.
\end{equation}
As before $\xi(x)\leq \theta' n\ll\lambda_{i,l}$ for $ x\notin A_n^\delta$, and $B_n\cap AN^\delta=\{x_{i,\tn}\}$. Therefore, by (\ref{ara1}) and the strong Markov property, for any $x\in B_n\setminus\{x_{i,\tn}\}$
\begin{equation}
\E_{x}[\exp(\int_0^{\tau_{x_{i,\tn}}} (\xi(X_s)-\lambda_{i,l} ) ds) \incl{\tau_{x_{i,\tn}}=\tau_{\Gamma_l}} \incl{\tau_{B_n^c}< \tau_{x_{i,\tn}}  }]\leq \cn.
\end{equation}
As a result, it is enough to check that
\begin{equation}
\sum_{x\in B_n\setminus\{x_{i,\tn}\}}\E_{x}[\exp(\int_0^{\tau_{x_{i,\tn}}} (\xi(X_s)-\lambda_{i,l} ) ds) \incl{\tau_{x_{i,\tn}}=\tau_{\Gamma_l}} \incl{\tau_{B_n^c}> \tau_{x_{i,\tn}}  }]\longrightarrow 0.
\end{equation}
We now construct a stochastic lower bound for $\tau_{x_{i,\tn}}$ through the Coupon collector's problem. Let $x$ be such that $d(x_{i,\tn},x)=r$. Without loss of generality we can assume that $x_{i,\tn}=(+1,+1,\dots,+1)$. Then the number of $-1$'s in the configuration of $x$ is exactly $r$. Now we reject all the jumps that switches a $+1$ to a $-1$, in other words, we consider only the spin sites with $-1$ and wait for all of them to become $+1$. Then this waiting time, denoted by $\tau'$, is a lower bound. Observe that the first time a $-1$ becomes a $+1$ is distributed as an exponential random variable with rate $\kappa r/n$ (recall that per spin the jump rate is $\kappa/n$); after this event there are now $r-1$ spins with a $-1$ sign and the first time one of them becomes a $+1$ is distributed as an exponential random variable with rate $\kappa (r-1)/n$; proceeding like this we wait finally an exponential time with rate $\kappa/n$ for the last $-1$ sign to become a $+1$. Hence, $\tau'$ is given by
\begin{equation}
\tau'=\frac{\alpha_r}{\kappa r/n}+\frac{\alpha_{r-1}}{\kappa (r-1)/n}+\cdots+\frac{\alpha_1}{\kappa/n}
\end{equation}
where $\alpha_1,\dots$ is an i.i.d. sequence of exponential random variables with rate 1. Hence, for any $x$ s.t. $d(x_{i,\tn},x)=r$
\begin{equation}
\begin{aligned}
& \E_{x}[\exp(\int_0^{\tau_{x_{i,\tn}}} (\xi(X_s)-\lambda_{i,l} ) ds) \incl{\tau_{x_{i,\tn}}=\tau_{\Gamma_l}} \incl{\tau_{B_n^c}> \tau_{x_{i,\tn}}  }]\\ &\leq
\E[\exp(-{\tau'} (\lambda_{i,l}-\theta' n))] 
\\&=\prod_{j=1}^r \frac{\kappa j/n}{\kappa j/n+\lambda_{i,l}-\theta' n}\leq \prod_{j=1}^r \frac{j}{Cn^2} 
\end{aligned}
\end{equation}
for some positive constant $C$. Note that for the last step once again we used that $\lambda_{i,l}\sim \theta n$.
For the last term above we use the following upper bound
\begin{equation}
\begin{aligned}
\prod_{j=1}^r \frac{j}{Cn^2} &= C'\exp\big(-2r \log n +\sum_{j=1}^r \log j\big)
\\& \leq C''\exp\big(-2r \log n +r \log r\big)=C'' n^{-2r}r^r.
\end{aligned}
\end{equation}
for some positive constants $C',C''$. Since
\begin{equation}
|\{x:\;d(x_{i,\tn},x)=r\}|=\binom{n}{r},
\end{equation}
in order to finish the proof it is enough to check that
\begin{equation}
\sum_{r=1}^{c n}n^{-2r} r^r \binom{n}{r}\longrightarrow 0.
\end{equation}
Since $c<1$, using Sterling's approximation we get
\begin{equation}
\begin{aligned}
\sum_{r=1}^{c n} n^{-2r} r^r \binom{n}{r}&\leq C \sum_{r=1}^{c n} 
n^{-2r}r^r \frac{n^n}{(n-r)^{n-r} r!}e^{-r}
\\& =C\sum_{r=1}^{c n}  n^{-r}\frac{r^r}{r!}\big[1+\frac{r}{n-r}\big]^{n-r}e^{-r}
\\&\leq C\sum_{r=1}^{c n}  n^{-r}\frac{r^r}{r!}
\end{aligned}
\end{equation}
Since $r^r\leq   r!\; e^r$ for any $r\in\N$ we get
\begin{equation}
\sum_{r=1}^{c n}  n^{-r}\frac{r^r}{r!}\leq \sum_{r=1}^{c n}  n^{-r}e^r\leq \sum_{r=1}^{\infty}  (e/n)^{r}=\frac{1}{1-e/n}-1\underset{n\to\infty}\longrightarrow 0.
\end{equation}
Hence, we are finished with the proof of part (i).
\end{proof}

\begin{proof}[Proof of Lemma \ref{lemmaAna} part (ii)]
	Since $i\not=k$ and $x_{k,\tn}\in A_n^\delta$ for $n$ large enough, $x_{k,\tn}\notin B_n$, where $B_n=B(x_{i,\tn},c n-3)$ as described in the proof of Lemma \ref{lemmaAna} part (i). Hence, via (\ref{ara1}) we have the right upper bound for $\nu_{i,l}(x_{k,\tn})$. 
	
For $l_n$ given in Assumption $(L)$ for some $\delta\in(0,1)$ we define
\begin{equation}
L_n:=\{x\in \Sigma_n:\;-l_n\leq \xi_n(x)\leq \psi_n(n\delta \log 2)\}
\end{equation}
and
\begin{equation}\label{proppn}
p_n:=P(\xi_n(x)\notin L_n).
\end{equation}
For $x\not= x_{i,\tn}$ we denote by $H(x)$ the number of nearest neighbour paths $x=y_0\to y_1\to\cdots \to y_d=x_{i,\tn}$, where $d=d(x,x_{i,\tn})$, such that, $\xi_n(y_j)\in L_n$ for all $j=1,\dots,d-1$. It is understood that $H(x)=1$ for $x$ such that $d(x,x_{i,\tn})=1$. Note that $H(x)$ and $H(z)$ have identical distributions for any $x,z$ with $d(x,x_{i,\tn})=d(z,x_{i,\tn})$. We label the expectation of any such distribution by $\mathcal{H}(d)$, that is, $\mathcal{H} (d)=E[H(x)]$ for some $x$ with $d(x,x_{i,\tn})=d$. Finally, we define
\begin{equation}
S(x):=\{y\sim x: d(y,x_i)=d(x,x_{i,\tn})-1 \mbox{ and }y\in L_n\}.
\end{equation}
Any nearest neighbour path from $x$ to $x_{i,\tn}$ of length $d(x,x_{i,\tn})$ that stays in $L_n$ in between should move to a vertex in $S(x)$ in its first step. Hence,
\begin{equation}
H(x)=\sum_{y\in S(x)} H(y).
\end{equation}
Since $\incl{y\in S(x)}$, $y\sim x$, and $H(y')$ are independent events for any $y'\sim x$ we get
\begin{equation}
E[H(x)]=E[|S(x)|]E[H(y)]
\end{equation}
For $x$, with $d=d(x,x_i)$, $|S(x)|$ is a Binomial random variable with parameters $d$ and $(1-p_n)$. This yields
\begin{equation}
\mathcal{H}(d)=d(1-p_n)\mathcal{H}(d-1),
\end{equation}
From $\mathcal{H}(1)=1$ it readily follows that
\begin{equation}
\mathcal{H}(d)=d! (1-p_n)^{d-1}.
\end{equation}
Since $H(x)\leq d!$, for any $\theta\in(0,1)$ we get
\begin{equation}
\begin{aligned}
\mathcal{H}(d)=E[H(x)]&=E[H(x)\incl{H(x)\leq \theta \mathcal{H}(d)}]+E[H(x)\incl{H(x)\geq \theta \mathcal{H}(d)}]
\\& \leq 
\theta \mathcal{H}(d)+d! P(H(x)\geq \theta \mathcal{H}(d)).
\end{aligned}
\end{equation}
Hence,
\begin{equation}
P(H(x)\geq \theta \mathcal{H}(d))\geq (1-\theta)\frac{\mathcal{H}(d)}{d!}=(1-\theta)(1-p_n)^{d-1}.
\end{equation}
Let $\theta=\theta_n=o(1)$. Since $d(x,x_i)\leq n$ for any $x\in\Sigma_n$
\begin{equation}
P(H(x)\leq \theta_n \mathcal{H}(d))\leq 1-(1-\theta_n)(1-p_n)^{n-1}\leq C \big(\theta_n+np_n+\theta_n n p_n\big).
\end{equation}
Recalling (\ref{proppn}) 
\begin{equation}
p_n=G(-l_n)+P(\xi_n(x)\geq \psi_n(n\delta\log 2)).
\end{equation}
The second term above decay as $2^{-\delta'n}$ for some $\delta'>0$. Hence, using Assumption (L) we have $\sum_n np_n<\infty$. Now we choose $\theta_n = n^{-1-a}$ for some $a>0$ and get
\begin{equation}
\sum_n \big(\theta_n+np_n+\theta_n n p_n\big) <\infty.
\end{equation}
Hence, by an application of Borel-Cantelli lemma we reach at that $P$-a.s. $H(x)\geq \theta_n H(r)$. Assumption $(R1)$ and part (i) of Lemma \ref{lemma1} imply $\xi_n(x_{l,\tn})\gg \psi_n(\delta \log 2 n)$. Thus, $P$-a.s. $L_n\cap \Gamma_l=\emptyset$ for $n$ large enough.

Let $d=d(x_{k,\tn},x_{i,\tn})$. Since $H(x_{k,\tn})\geq n^{-1-a} \mathcal{H}(d)=n^{-1-a} d! (1-p_n)^{d-1}$ and the probability of any nearest neighbour path of length $d$ is $n^{-d}$, using the probabilistic representation of $\nu_{i,l}$ we get that $P$-a.s. for $n$ large enough
\begin{equation}\label{busonson}
\begin{aligned}
\nu_{i,l}(x_{k,\tn})&=\E_{x_{k,\tn}}[\exp(\int_0^{\tau_{x_{i,\tn}}}    (\xi_n(X_s)-\lambda_{i,l}) ds\incl{\tau_{x_{i,\tn}}=\tau_{\Gamma_l}})]
\\&\geq \frac{n^{-1-a}d!(1-p_n)^{d-1}}{n^d}\prod_{j=0}^{d-1} \frac{\kappa}{\kappa+\lambda_{i,l}-\xi_n(y_j)}
\\&\geq \frac{n^{-1-a}d!(1-p_n)^{d-1}}{n^d} \frac{\kappa}{\kappa+\lambda_{i,l}-\xi_{k,\tn}}\left[\frac{\kappa}{\kappa+\lambda_{i,l}+l_n}\right]^{d-1}.
\end{aligned}
\end{equation}
$d=dist(x_i,x_k)\sim n/2$ by part (ii) of Lemma \ref{lemma2}. This and the fact that $p_n\to 0$ yield for the first term above
\begin{equation}
\frac{n^{-1-a}d!(1-p_n)^{d-1}}{n^d}=\exp(-O(n)).
\end{equation}
$\lambda_{i,l}+\kappa=\xi_{i,\tn}+o(1)$ by Lemma \ref{lemmaAna2}; $\xi_{i,\tn}\sim \theta n$ by part (i) of Lemma \ref{lemma1}; by part (ii) of Lemma \ref{lemma1} $\xi_{i,\tn}-\xi_{k,\tn}=C+o(1)$ for some random positive constant $C$. Hence, using the fact that $d=dist(x_i,x_k)\sim n/2$  and  $l_n\ll n$ we can conclude that
\begin{equation}
\begin{aligned}
\nu_{i,l}(x_{k,\tn})&\geq C' \exp\Big(-d\log (\xi_{i,\tn}+o(1)+l_n ) +O(n) \Big)
\\&=\exp\Big(-\frac{n}{2}\log n (1+o(1))\Big).
\end{aligned}
\end{equation}
This gives the right lower bound and we are finished with the proof of part (ii) of Lemma \ref{lemmaAna}.

\end{proof}

\section*{Acknowledgments}

This work was supported by German Research Foundation (DFG), within the SPP Priority Programme 1590 ``Probabilistic Structures in Evolution". 

\bibliographystyle{plain}
\bibliography{PAM}

\begin{thebibliography}{10}

\bibitem{BAC08}
G.~Ben~Arous and J.~{\v{C}}ern{\'y}.
\newblock The arcsine law as a universal aging scheme for trap models.
\newblock {\em Comm. Pure Appl. Math.}, 61(3):289--329, 2008.

\bibitem{CK70}
J.F. Crow and M.~Kimura.
\newblock {\em An introduction to population genetics theory}.
\newblock Harper \& Row, New York, 1970.

\bibitem{VK14}
J.~de~Visser and J.~Krug.
\newblock Empirical fitness landscapes and the predictability of evolution.
\newblock {\em Nat. Rev. Genet.}, (15):480--490, 2014.

\bibitem{D81}
B.~Derrida.
\newblock Random-energy model: an exactly solvable model of disordered systems.
\newblock {\em Phys. Rev. B (3)}, 24(5):2613--2626, 1981.

\bibitem{Eig71}
M.~Eigen.
\newblock Self-organization of matter and the evolution of macromolecules.
\newblock {\em Naturwissenschaften}, 58(10):465--523, 1971.

\bibitem{F66}
W.~Feller.
\newblock {\em An introduction to probability theory and its applications.
  {V}ol. {II}}.
\newblock John Wiley \& Sons, Inc., New York-London-Sydney, 1966.

\bibitem{FM90}
K.~Fleischmann and S.~A. Molchanov.
\newblock Exact asymptotics in a mean field model with random potential.
\newblock {\em Probab. Theory Related Fields}, 86(2):239--251, 1990.

\bibitem{FP97}
S.~Franz and L.~Peliti.
\newblock Error threshold in simple landscapes.
\newblock {\em J. Phys. A}, 30(13):4481--4487, 1997.

\bibitem{FPS93}
S.~Franz, L.~Peliti, and M.~Sellitto.
\newblock An evolutionary version of the random energy model.
\newblock {\em J. Phys. A}, 26(13):L1195, 1993.

\bibitem{BG99}
W.~Gabriel and E.~Baake.
\newblock {\em {Biological evolution through mutation, selection, and drift: An
  introductory review}}.
\newblock 1999.

\bibitem{GKM07}
J.~G\"artner, W.~K\"onig, and S.~Molchanov.
\newblock Geometric characterization of intermittency in the prabolic anderson
  model.
\newblock {\em Ann. Probab.}, 35(2):439--499, 2007.

\bibitem{GM90}
J.~G{\"a}rtner and S.~A. Molchanov.
\newblock Parabolic problems for the {A}nderson model. {I}. {I}ntermittency and
  related topics.
\newblock {\em Comm. Math. Phys.}, 132(3):613--655, 1990.

\bibitem{GM98}
J.~G{\"a}rtner and S.~A. Molchanov.
\newblock Parabolic problems for the {A}nderson model. {II}. {S}econd-order
  asymptotics and structure of high peaks.
\newblock {\em Probab. Theory Related Fields}, 111(1):17--55, 1998.

\bibitem{K78}
J.~F.~C. Kingman.
\newblock A simple model for the balance between selection and mutation.
\newblock {\em J. Appl. Probability}, 15(1):1--12, 1978.

\bibitem{W16}
W.~K\"onig.
\newblock {\em The parabolic Anderson model, random walk in random potential}.
\newblock Birkh\"auser, Basel, 2016.

\bibitem{KLMS09}
W.~K{\"o}nig, H.~Lacoin, P.~M{\"o}rters, and N.~Sidorova.
\newblock A two cities theorem for the parabolic {A}nderson model.
\newblock {\em Ann. Probab.}, 37(1):347--392, 2009.

\bibitem{M76}
P.~A.~P. Moran.
\newblock Global stability of genetic systems governed by mutation and
  selection.
\newblock {\em Math. Proc. Cambridge Philos. Soc.}, 80(2):331--336, 1976.

\bibitem{MOS11}
P.~M{\"o}rters, M.~Ortgiese, and N.~Sidorova.
\newblock Ageing in the parabolic {A}nderson model.
\newblock {\em Ann. Inst. Henri Poincar\'e Probab. Stat.}, 47(4):969--1000,
  2011.

\bibitem{HMS08}
R.~van~der Hofstad, P.~M{\"o}rters, M.~Ortgiese, and N.~Sidorova.
\newblock Weak and almost sure limits for the parabolic anderson model with
  heavy-tailed potentials.
\newblock {\em Ann. Appl. Probab.}, 18(6):2450--2494, 2008.

\bibitem{W32}
S.~Wright.
\newblock The roles of mutation, inbreeding, crossbreeding and selection in
  evolution.
\newblock {\em Proceeding of the sixth international congress of genetics},
  1:356--366, 1932.

\end{thebibliography}

\end{document}